\documentclass[12 pt]{amsart}

\usepackage{hyperref}
\usepackage{etex}
\usepackage[shortlabels]{enumitem}
\usepackage{amsmath}
\usepackage{amsxtra}
\usepackage{amssymb}
\usepackage{enumitem}

\usepackage{amscd}
\usepackage{amsthm}
\usepackage{adjustbox}
\usepackage{amsfonts}
\usepackage{amssymb}
\usepackage{eucal}
\usepackage[all]{xy} 
\usepackage{graphicx}
\usepackage{tikz-cd}
\usepackage{mathrsfs}
\usepackage{subfiles}
\usepackage{mathpazo}
\usepackage[colorinlistoftodos, textsize=tiny]{todonotes}
\usepackage{morefloats}
\usepackage{pdfpages}
\usepackage{thm-restate}
\usepackage[percent]{overpic}
\usepackage[utf8]{inputenc}
\usepackage{epigraph}
\usepackage{csquotes}
\usepackage[margin=1in]{geometry}
\usepackage{adjustbox}
\usepackage{microtype}
\usepackage{stmaryrd}
\usepackage{svg}

\usepackage{bm}
\usepackage{verbatim}
\usepackage{stmaryrd}
\usepackage{scalerel}
\usepackage{stackengine}
\stackMath
\newcommand\reallywidehat[1]{%
\savestack{\tmpbox}{\stretchto{%
  \scaleto{%
    \scalerel*[\widthof{\ensuremath{#1}}]{\kern-.6pt\bigwedge\kern-.6pt}%
    {\rule[-\textheight/2]{1ex}{\textheight}}
  }{\textheight}%
}{0.5ex}}%
\stackon[1pt]{#1}{\tmpbox}%
}
\usepackage{mathtools}

\graphicspath{ {images/} }

\RequirePackage{color}
\definecolor{myred}{rgb}{0.75,0,0}
\definecolor{mygreen}{rgb}{0,0.5,0}
\definecolor{myblue}{rgb}{0,0,0.65}

\usepackage{color}

\usepackage{hyperref}
\hypersetup{citecolor=blue}
\usepackage{tikz}
\usetikzlibrary{calc, arrows.meta, decorations.markings}

\theoremstyle{plain}
\newtheorem{theorem}[subsubsection]{Theorem}

\newtheorem{proposition}[subsubsection]{Proposition}
\newtheorem{prop/constr}[subsubsection]{Proposition/Construction}
\newtheorem{lemma}[subsubsection]{Lemma}
\newtheorem{corollary}[subsubsection]{Corollary}

\theoremstyle{definition}
\newtheorem{definition}[subsubsection]{Definition}
\newtheorem{remark}[subsubsection]{Remark}

\newtheorem{conjecture}[subsubsection]{Conjecture}

\newtheorem*{claim*}{Claim}
\theoremstyle{remark}

\numberwithin{equation}{section}
\newcommand\nc{\newcommand}
\nc\on{\operatorname}
\nc\renc{\renewcommand}

\DeclareMathOperator\Hom{Hom}

\DeclareMathOperator\SL{SL}
\DeclareMathOperator\PGL{PGL}

\DeclareMathOperator\tr{tr}

\nc\mf\mathfrak
\nc\mc\mathcal
\nc\mb\mathbb
\nc\msf\mathsf
\nc\mscr\mathscr

\newcommand{\Qbar}{\overline{\mathbb{Q}}}

\newcommand{\Zbar}{\overline{\mathbb{Z}}}

\title{Density of integral points in the Betti moduli of quasi-projective varieties}
\author{Simone Coccia and Daniel Litt}
\date{\today}

\begin{document}

\begin{abstract}
	Let $Y$ be a smooth quasi-projective complex variety equipped with a simple normal crossings compactification. We show that integral points are potentially dense in the (relative) character varieties parametrizing   $\SL_2$-local systems on $Y$ with fixed algebraic integer traces along the boundary components. The proof proceeds by using work of Corlette-Simpson to reduce to the case of Riemann surfaces, where we produce an integral point with Zariski-dense orbit under the mapping class group.
\end{abstract}

\maketitle


\section{Introduction}\label{section:introduction}
\subsection{Conjecture and main result}
 Let $Y$ be a smooth complex variety equipped with a smooth projective simple normal crossings compactification $\overline{Y}$, with $D=\overline{Y}\setminus Y$. Given a commutative ring $R$ and an affine algebraic group $G/R$, the \emph{$G$-representation variety} $$\on{Hom}(\pi_1(Y), G)$$ is the affine $R$-scheme whose $S$-points for an $R$-algebra $S$ are $$\on{Hom}(\pi_1(Y), G)(S):=\on{Hom}(\pi_1(Y), G(S)).$$ The \emph{$G$-character variety} of $Y$ is the (categorical) quotient $$X_G(Y):=\on{Hom}(\pi_1(Y), G)/G,$$ where $G$ acts by conjugation.
 
 For each component $D_i$ of $D$, $i=1, \cdots ,n$, fix a small loop $\gamma_i$ around $D_i$ and an $R$-point $C_i$ in the adjoint quotient $(G/_{\text{ad}}G)(R)$. There is a natural map $$p_D: X_G(Y)\to (G/_{\text{ad}}G)^n$$ induced by the map $$\rho\mapsto (\rho(\gamma_i))_{i=1, \cdots, n}.$$ Setting $\underline{C}=(C_i)_{i=1,\cdots, n}$ we define the \emph{relative character variety} $$X_{G,\underline{C}}(Y):=p_D^{-1}(\underline{C}).$$ For example, if $G=\SL_2$, $X_{G,\underline{C}}(Y)$ parametrizes conjugacy classes of representations $\rho$ of $\pi_1(Y)$ into  $\SL_2$ with $\on{tr}(\rho(\gamma_i))$ fixed.
 
 The goal of this paper is to provide some evidence for the following conjecture.
\begin{conjecture}\label{conj:main-conjecture}
	Let $G$ be a Chevalley group over $\mathbb{Z}$, $K$ a number field, and $\mathscr{O}_K$ the ring of integers of $K$. Fix $\underline{C}\in (G/_{\text{ad}}G)(\mathscr{O}_K)^n$. Then integral points are potentially Zariski-dense in the $K$-scheme $X_{G, \underline{C}}(Y)_K.$ That is, there exists a finite extension $L/K$ such that the Zariski-closure of the $\mathscr{O}_L$-points of $X_{G, \underline{C}}(Y)$ contains $X_{G, \underline{C}}(Y)_K.$
\end{conjecture}
Recall that here a Chevalley group is a smooth affine group scheme over $\mathbb{Z}$ with connected reductive fibers, admitting a fiberwise maximal $\mathbb{Z}$-torus; for example,  $\SL_{n, \mathbb{Z}},\PGL_{n, \mathbb{Z}}$, $\on{Sp}_{2n, \mathbb{Z}}$ are Chevalley. Some version of this conjecture has been considered by a number of people; for example see \cite[Question 5.4.3(2)]{litt2024motives}.

In this paper we verify \autoref{conj:main-conjecture} for $G=\SL_2,\PGL_2$:

\begin{theorem}\label{thm:main-theorem}
	Let $G=\SL_{2, \mathbb{Z}}$ or  $\PGL_{2,\mathbb{Z}}$. Fix a number field $K$ and $\underline{C}\in (G/_{\text{ad}}G)(\mathscr{O}_K)^n.$ Then integral points are potentially dense in $X_{G,\underline{C}}(Y).$
\end{theorem}

The proof proceeds by reduction to the case of curves, relying on Corlette-Simpson's and Loray-Pereira-Touzet's classification of rank $2$ local systems on quasi-projective varieties \cite{corlette2008classification, loray2016representations}. We handle the case where $Y$ is a curve (say, of genus $g$ with $n\geq 0$ punctures) by constructing, for every $\underline{C}$, an integral representation whose orbit under the pure mapping class group of a surface of genus $g$ with $n$ punctures is Zariski-dense in $X_{G, \underline{C}}(Y)$. As the action of the mapping class group preserves integrality, this suffices. 

Our results on Zariski-density of integral points in (relative) character varieties of surface groups are \autoref{theorem:main-theorem-sl2} and  \autoref{theorem:pgl2-density}. In particular, we show that if $K$ is the field of definition of $\underline{C}$, then there exists a degree $4$ extension $L/K$ such that $\mathscr{O}_L$-points are Zariski-dense in the relative $\SL_2$-character variety of a curve of genus $g$ with $n$ punctures; see \autoref{rem:degree-4}. Some such field extension is necessary; see \autoref{rem:extension-necessary}.

\subsection{Motivation and related work}
The primary antecedent to \autoref{conj:main-conjecture} is Simpson's conjecture on integrality of rigid local systems \cite[p.~9]{simpson1992higgs}, which is precisely the statement that integral points are Zariski-dense in $0$-dimensional components of $X_{G, \underline{C}}(Y)$, at least when $Y$ is projective. Even this case and its quasi-projective variant is open, though beautiful work of Esnault-Groechenig (in the case $G=GL_n$) \cite{esnault2018cohomologically} and Klevdal-Patrikis (for general $G$) \cite{klevdal-patrikis} prove that \emph{reduced} isolated points of $X_{G, \underline{C}}(Y)$ are integral, for $\underline{C}$ quasi-unipotent. De Jong-Esnault \cite{de2024integrality} show that, if non-empty, $X_G(Y)$ has a $\overline{\mathbb{Z}_\ell}$-point for every $\ell$ (and much more); this would evidently be a consequence of \autoref{conj:main-conjecture}. Under mild hypotheses they show $X_G(Y)$ has a $\overline{\mathbb{Z}}$-point. All of these results rely on the existence of \emph{$\ell$-adic companions}, due to Lafforgue \cite{lafforgue2002chtoucas} in dimension one and Drinfeld \cite{drinfeld2012conjecture} in general, ultimately relying on Lafforgue's work on the Langlands program for function fields over finite fields.

Why might one believe \autoref{conj:main-conjecture}? Aside from the fact that it generalizes Simpson's conjecture to positive-dimensional components of $X_{G, \underline{C}}(Y)$, it is also motivated by a conjecture of Campana \cite[Conjecture $13.23$]{campana2011special} predicting which varieties should have a potentially Zariski-dense set of $S$-integral points. One particular instance of such conjectures is that log Calabi-Yau varieties\footnote{A variety $Z$ is log Calabi-Yau if it admits a normal projective compactification $X$ with reduced boundary divisor $D$ such that $K_X+D\sim 0$.} admit an integral model with a Zariski-dense set of integral points (see \cite[Théorème $7.7$]{campana2011special}).
An expectation attributed to Kontsevich-Soibelman is that in many cases character varieties are \lq\lq cluster varieties\rq\rq{}, hence log Calabi-Yau (see the discussion after Conjecture $5$ of \cite{simpson2015harmonic}). Campana's conjecture then predicts that they should have a Zariski dense set of integral points. Whang \cite[Theorem 1.1]{whang2020global} has proven that relative  $\SL_2$-character varieties of surfaces are log Calabi-Yau, so that our \autoref{theorem:main-theorem-sl2} answers positively Campana's conjecture for such varieties. We remark that our result is stronger than the expectation of Campana's conjecture, as we prove potential density of integral points rather than $S$-integral ones.

Our results on Zariski-density of integral points for surface groups are closely related to the study of mapping class group dynamics on (relative) character varieties; indeed, we prove density by finding integral points with Zariski-dense mapping class group orbit. We rely on the study of the geometry of character varieties from \cite{whang:ant}. Recently Golsefidy-Tamam \cite{golsefidy-complete} (see also \cite{golsefidy2025closure} for a summary of results) have closely studied Zariski-density of mapping class group orbits in character varieties of surfaces; we expect we could have used their results for our purposes as well, though we have opted for a more self-contained exposition. In general, dynamics of mapping class groups on character varieties has been studied from a number of points of view by Goldman \cite{goldman2005mapping, goldman2009ergodicity}, Previte-Xia \cite{previte2000topological, previte2002topological}, and others, including the second-named author and collaborators \cite{lam2023finite, landesman2024canonical}.

Arithmetic aspects of $\SL_2$-character varieties of surfaces have recently been studied in the work of several authors. For instance, strong approximation results for surfaces of Markoff type (which are relative character varieties of the projective line with four punctures) have been established in the work of Bourgain-Gamburd-Sarnak \cite{bourgain2016markoff} and Chen \cite{chen2024nonabelian} (see also \cite{martin2025new} for a more elementary approach to part of Chen's work). Ghosh and Sarnak \cite{ghosh2022hasse} investigated the integral Hasse principle for a family of Markoff cubic surfaces, showing (among various things) that almost all surfaces admitting a $\mb{Z}_p$ solution at all primes contain a Zariski dense set of integral points. Whang \cite{whang:ant, whang:israel-journal} obtained a structure theorem for integral points on relative $\SL_2$-character varieties of surfaces by means of mapping class group descent, and applied this to the effective determination of integral points on curves in these varieties.

The second-named author will use the potential-density results proven here for some applications to the Ekedahl--Shepherd-Barron--Taylor conjecture for isomonodromy foliations on relative moduli of flat connections, in upcoming work with Yeuk Hay Joshua Lam, building on \cite{lam2025algebraicity}. From this point of view the potential density studied here is a ``non-abelian" analogue of the integral structure on singular cohomology. See \cite[\S5]{litt2024motives} for some philosophical discussion along these lines.

\subsection{Acknowledgments}
To be added after referee process is complete.

\section{Notation}

We will use the following notation:
\begin{itemize}
	\item $\Sigma_{g,n}$ is a smooth, orientable (topological) surface of genus $g$ with $n$ punctures.
	\item $\Gamma_{g,n}$ is the pure mapping class group of $\Sigma_{g,n}$, i.e.~the component group of the group of orientation-preserving homeomorphisms of $\Sigma_{g,n}$ that fix each puncture pointwise, equipped with the compact-open topology.
	\item Given a simple closed curve $a$ in $\Sigma_{g,n}$, $\tau_a$ denotes the Dehn twist along $a$, viewed as an element of the mapping class group of $\Sigma_{g,n}$;
	\item for a set of simple closed curves $\mc A=\{a_i\}_{i \in \mc I}$ in $\Sigma_{g,n}$, we denote by $\Gamma_{\mc A}$ the subgroup of the mapping class group of $\Sigma_{g,n}$ generated by $\{\tau_a \, \vert \, a \in \mc A \}$;
	\item $\mu_{\infty} \subset \Qbar$ is the set of roots of unity, and $E:=2\Re(\mu_\infty)$ is the set of real numbers of the form $\zeta+\zeta^{-1}$, for $\zeta\in \mu_\infty$.
\end{itemize}

\section{Dynamics on relative character varieties}
In  this section we will collect some results regarding the dynamics of mapping class group actions on relative  $\SL_2$-character varieties of surface groups. Except for \autoref{proposition:zariski-dense-orbit}, the material of this section is mostly recalled from \cite{whang:ant}. Ultimately we will apply these results to prove \autoref{thm:main-theorem} in the case of algebraic curves in \autoref{sec:potential-density-sl2} (for  $\SL_2$) and \autoref{sec:pgl2} (for $\PGL_2$). We will explain how to deduce \autoref{thm:main-theorem} from this case in \autoref{sec:corlette-simpson}.

Throughout this section we set $E\subset \overline{\mathbb{Q}}$ to be the set $E:=2\Re(\mu_\infty)$, i.e.~the set of real numbers of the form $\zeta+\zeta^{-1}$ for $\zeta$ a root of unity.

\subsection{Geometry of relative character varieties}

Let $\Sigma_{g,n}$ be an orientable topological surface of genus $g$, with $n$ punctures. Set $X_{g,n}:=X_{\SL_2}(\Sigma_{g,n})$ to be the $\SL_2$-character variety of $\pi_1(\Sigma_{g,n})$. The adjoint quotient of $SL_{2, \mathbb{Z}}, (SL_{2, \mathbb{Z}}/_{\text{ad}} SL_{2, \mathbb{Z}})$ is naturally isomorphic to the affine line over $\mathbb{Z}$ via the trace map. We set $k=(k_1, \cdots, k_n)\in \mathbb{A}^n(\overline{\mb Z})$ to be a tuple of algebraic integers and set $X_{g,n, k}$ to be the relative character variety of $\Sigma_{g,n}$, parametrizing $\SL_2$-representations of $\pi_1(\Sigma_{g,n})$ with trace $k_i$ about the $i$-th puncture of $\Sigma_{g,n}$.

Let $\mathcal{P}=a_1\cup \cdots\cup a_{3g-3+n}$ be a pants decomposition of $\Sigma_{g,n}$, i.e.~a collection of $3g-3+n$ simple closed curves in $\Sigma_{g,n}$ whose complement $\Sigma\setminus \mathcal{P}$ is homeomorphic to a disjoint union of copies of $\Sigma_{0,3}$. This induces a map $\tr_{\mathcal{P}} \colon X_{g,n, k} \to \mb{A}^{3g-3+n}$ given by the traces $\tr_{a_i}$ along the paths $a_i$ of $\mathcal{P}$. For $t=(t_1, \cdots, t_{3g-3+n}) \in \mb{A}^{3g-3+n}(\overline{\mathbb{Q}})$, we denote $X^\mathcal{P}_{k,t} \coloneqq \tr_{\mc P}^{-1}(t)$. Then the subgroup $\Gamma_{\mathcal{P}}$ of the mapping class group $\Gamma_{g,n}$ of $\Sigma_{g,n}$ generated by Dehn twists $\tau_{a_i}$ about the paths $a_i$ in $\mathcal{P}$ is a free abelian subgroup of the mapping class group whose action on $X_{g,n}$ (via outer automorphisms of $\pi_1(\Sigma_{g,n})$) preserves $X^\mathcal{P}_{k,t}$.

\subsubsection{Character varieties of $\Sigma_{1,1}$ and $\Sigma_{0,4}$} \label{subsec:1104}

We first give a description of the relative character varieties of $\Sigma_{1,1}$ and $\Sigma_{0,4}$, which turn out to be affine cubic surfaces of Markoff type. We refer to \cite[\S $2.3$]{whang:ant} for more details.

We first deal with $\Sigma_{1,1}$. Let $(\alpha, \beta, \gamma)$ be an optimal sequence of generators (see \cite[\S2A1]{whang:ant}) for $\pi_1 (\Sigma_{1,1})$, where $\gamma$ is a loop around the puncture. The map $(\tr_{\alpha}, \tr_{\beta}, \tr_{\alpha\beta}) \colon X_{1,1} \to \mb A^3$ is an isomorphism. We have the identity:
\begin{equation*}
	\tr_{\gamma} = \tr_{\alpha}^2 +  \tr_{\beta}^2+  \tr_{\alpha\beta}^2 -  \tr_{\alpha} \tr_{\beta}\tr_{\alpha \beta} -2.
\end{equation*}
Then, writing $(x,y,z)=(\tr_{\alpha}, \tr_{\beta}, \tr_{\alpha\beta})$, the relative character variety $X_{1,1,k}$ is defined by the cubic equation
\begin{equation*}
	x^2+y^2+z^2-xyz-2=k.
\end{equation*}
Let us now deal with $\Sigma_{0,4}$. Let $\gamma_1, \gamma_2,\gamma_3, \gamma_4$ be an optimal sequence of generators, where each $\gamma_i$ is a loop around the corresponding puncture. Let $k=(k_1,k_2,k_3,k_4) \in \mb A^4(\mb C)$ and set $(x,y,z)=(\tr_{\gamma_{1}\gamma_{2}},\tr_{\gamma_{2}\gamma_{3}},\tr_{\gamma_{1}\gamma_{3}})$. Then $X_{0,4,k}$ is defined by the cubic equation:
\begin{equation*}
	x^{2}+y^{2}+z^{2}+xyz=Ax+By+Cz+D
\end{equation*}
where
\begin{equation*}
	\begin{cases}
		A=k_{1}k_{2}+k_{3}k_{4}\\
		B=k_{1}k_{4}+k_{2}k_{3}\\
		C=k_{1}k_{3}+k_{2}k_{4}
	\end{cases}
	\quad \text{and} \quad D=4-\sum_{i=1}^{4}k_{i}^{2}-\prod_{i=1}^{4}k_{i}.
\end{equation*}   
In both cases we have the following:
\begin{lemma}\label{lemma:dominant}
	Let $t \in \mb C$ and let $X=X_{1,1,k}$ (resp. $X=X_{0,4,k}$). Let $\pi_y \colon X \to \mb A^1$ be the projection map $\pi_y(x,y,z)=y$, which coincides with the trace map $\tr_{\beta}$ (resp. $\tr_{\gamma_2\gamma_3}$). Then the restriction of $\pi_y$ to the curve $x=t$ is dominant. 
\end{lemma}
We also collect here the following facts:
\begin{proposition}\label{proposition:degenerate-character-varieties}
	We have that:
	\begin{itemize}
		\item the character variety of the torus $X_{1,0}$ is defined by $x^2+y^2+z^2-xyz-4=0$;
		\item there is a single $\SL_2$-representation of $\pi_1(\Sigma_{1,1})$ up to conjugacy with monodromy $-I$ at the puncture, and it corresponds to the point $(0,0,0)$ of the Markoff surface $X_{1,1,-2} \colon x^2+y^2+z^2-xyz=0$.
	\end{itemize}
\end{proposition}
\begin{proof}
	See \cite[Theorem $6.3$]{munoz2024coordinate} and \cite[Section $4.2$]{logares2013hodge}.
\end{proof}

\subsubsection{Dynamics of relative character varieties}

\begin{definition}\label{definition:perfect-fiber}
	Let $\mc P$ be a pants decomposition of $\Sigma=\Sigma_{g,n}$, $k \in \mb A^n$ and $t \in \mb A^{3g-3+n}$. We say that $X^\mathcal{P}_{k,t}$ is \emph{perfect} if
	\begin{itemize}
		\item for all $a_i\in \mc P$, we have $\tr_{a_i}(X^\mathcal{P}_{k,t}) \neq \pm 2$  and
		\item for each $[\rho]$ in $X_{k,t}^{\mathcal{P}}(\mathbb{C})$, its restriction to each component of ${\Sigma \setminus \mc P}$ is irreducible, or $(g,n,k)=(1,1,2)$.
	\end{itemize}
\end{definition}
\begin{remark}\label{rem:P-good}
Note that both conditions above are really only conditions on $t$. For the first condition this is clear; for the second, it follows as an $\SL_2$-local system on $\Sigma_{0,3}$ is determined up to semisimplification by its 	three boundary traces. In particular (see \cite[Lemma 3.3]{whang:ant}) it is irreducible unless the three boundary traces $x,y,z$ satisfy $$x^2+y^2+z^2-xyz=4.$$

Note that the set of $t$ such that $X_{k,t}^{\mathcal{P}}$ is not perfect is a proper Zariski-closed subset of $\mathbb{A}^{3g-3+n}$.
\end{remark}

Let $X^\mathcal{P}_{k,t}$ be a perfect fiber. Fix $(\lambda_1, \dots, \lambda_{3g-3+n}) \in (\Qbar^{\times})^{3g-3+n}$ such that $\lambda_i+\lambda_i^{-1}=t_i$. We denote by $T_{z_i} \colon \mb{G}_m^{3g-3+n} \to \mb{G}_m^{3g-3+n}$ the map given by multiplication of the $i$-th coordinate by $\lambda_i$. We recall the following result from \cite[Proposition $4.3$]{whang:ant}:
\begin{proposition}[Whang]\label{proposition:whang-perfect-fibers}
	If $X^\mathcal{P}_{k,t}$ be a perfect fiber, then there is a morphism
	\begin{equation*}
		F \colon \mb{G}_m^{3g-3+n} \to X^\mathcal{P}_{k,t}
	\end{equation*}
	defined over $\Qbar$ satisfying the following:
	\begin{enumerate}
		\item at the level of $\Qbar$-points, $F$ is surjective with finite fibers,
		\item the action of $T_{z_i}$ on $\mb{G}_m^{3g-3+n}$ lifts the action of the Dehn twist $\tau_{a_i}$ on $X^\mathcal{P}_{k,t}$.
	\end{enumerate}
\end{proposition}

Recall that $E=\{\lambda + \lambda^{-1} \, \vert \, \lambda \in \mu_{\infty}\}$, where $\mu_{\infty}$ is the set of all roots of unity. Notice that, if $K$ is a number field, then $E \cap K$ is a finite set.

\begin{lemma}\label{lemma:zariski-dense-fiber}
	Let $\mc P$ be a pants decomposition of $\Sigma$ and let $p \in X$ be a point contained in a perfect fiber $X^{\mc{P}}_{k,t}$ of $\tr_{\mc P}$. If $t \in (\mb A^1(\overline{\mathbb{Q}}) \setminus E)^{3g-3+n}$, then $\Gamma_{\mathcal{P}} \cdot p$ is Zariski dense in $X^{\mc{P}}_{k,t}$.
\end{lemma}
\begin{proof}
	Since the monodromy along $a_i$ has infinite order (by the assumption that no $t_i$ lies in $E$) and $\tr(a_i) \neq \pm 2$, the eigenvalues of the monodromy along $a_i$ must have infinite multiplicative order. By \autoref{proposition:whang-perfect-fibers}, the orbit of any point of $\mb{G}_m^{3g-3+n}$ under the $\langle T_{z_i}\rangle_{i=1,\cdots, 3g-3+n}$-action lifting the $\Gamma_{\mc P}$-action on $X_{k,t}^{\mc P}$ is Zariski dense. The claim follows from the surjectivity of  $\mb{G}_m^{3g-3+n}(\Qbar) \to X^{\mc{P}}_{k,t}(\Qbar)$.
\end{proof}


\begin{definition}\label{defn:p-good}
	Given a pants decomposition $\mc P$ of $\Sigma$, we say that $[\rho] \in X_{g,n,k}(\overline{\mathbb{Q}})$ is $\mc P$-\emph{good} if $t \coloneqq \tr_{\mc P}([\rho]) \in (\mb A^1(\overline{\mathbb{Q}})\setminus E)^{3g-3+n}$ and $X^{\mc P}_{k,t}$ is a perfect fiber.
\end{definition}

The following proposition will be our main tool for showing pure mapping class group orbits are Zariski-dense in relative character varieties.

\begin{proposition}\label{proposition:zariski-dense-orbit}
	Let $\mc P$ be a pants decomposition and let $p \in X_{g,n,k}(\Qbar)$ be a $\mc P$-good point. Then $\Gamma_{g,n} \cdot p$ is Zariski dense in $X_{g,n,k, \overline{\mathbb{Q}}}$.
\end{proposition}
\begin{proof}
	Let $K$ be a number field containing the fields of definition of $p$ and $X_{g,n,k}$. In particular, we have that $\Gamma_{g,n} \cdot p \subseteq X_{g,n,k}(K)$. Let $\mc P = a_1 \cup \cdots \cup a_{3g-3+n}$ be a pants decomposition of $\Sigma_{g,n}$. 
	
	It is sufficient to prove that $\tr_{\mc P}(\Gamma_{g,n}\cdot p)$ is Zariski dense in $\mb{A}^{3g-3+n}$. Suppose this is the case. The set of $t \in \mb{A}^{3g-3+n}$ such that $X^{\mc{P}}_{k,t}$ is \emph{not} perfect is a proper Zariski closed subset of $\mb{A}^{3g-3+n}$ (see \autoref{rem:P-good}), and, since $E \cap K$ is finite, the same is true for the set of $t \in \mb{A}^{3g-3+n}$ such that at least one of the coordinates of $t$ lies in $E \cap K$. Thus, if $\tr_{\mc P}(\Gamma_{g,n}\cdot p)$ is Zariski dense in $\mb{A}^{3g-3+n}$, $\tr_{\mc P}(\Gamma_{g,n}\cdot p)$ would contain a Zariski dense set of $t \in (\mb A^1 \setminus E)^{3g-3+n}$ for which $X^{\mc{P}}_{k,t}$ is perfect, so the desired Zariski-density statement for $X_{g,n,k, \overline{\mathbb{Q}}}$ would follow from \autoref{lemma:zariski-dense-fiber}.
	
	We now show that $\tr_{\mc P}(\Gamma_{g,n}\cdot p)$ is Zariski dense in $\mb{A}^{3g-3+n}$.
	
	Let $\Sigma_i$ be the surface of type $(g',n')=(0,4)$ or $(1,1)$ obtained by gluing the components of $\Sigma\setminus \mathcal{P}$ bounded by $a_i$ along $a_i$. We have a natural restriction map $X_{k,t}^{\mc P}(\Sigma) \to X_{g', n', k_i'}(\Sigma_i)$, where $k_i'$ is the vector of traces naturally induced on the boundary of $\Sigma_i$.
 Moreover, the restriction of $p$ belongs to $X_{g', n', k_i'}(\Sigma_i)$ and is $\{a_i\}$-good with respect to the pants decomposition $\{a_i\}$ of $\Sigma_i$. Evidently $ X_{g', n', k_i'}(\Sigma_i)$ is defined over $K$.
	
	Let $b_i$ be a simple essential curve in $\Sigma_i$ such that $i(a_i,b_i)=1$ if $\Sigma_i$ is of type $(1,1)$ and such that $i(a_i,b_i)=2$ if $\Sigma_i$ is of type $(0,4)$, where $i(a,b)$ denotes the intersection number, as in \autoref{fig:curve-choice}. Recall that, for a set of simple closed curves $\mc C$, we denote by $\Gamma_{\mc C}$ the subgroup of the mapping class group generated by $\{\tau_a \, \vert \, a \in \mc C \}$, the Dehn twists about curves in $\mc C$. Moreover, we denote by $\tr_{\mc C} \colon X_{g,n,k} \to \mb A^{\mc C}$ the map given by the traces along all $a \in \mc C$. 
	
	\begin{figure}

\tikzset{every picture/.style={line width=0.75pt}} 

\begin{tikzpicture}[x=0.75pt,y=0.75pt,yscale=-1,xscale=1]

\draw    (140.57,111.69) .. controls (141.36,141.51) and (191.69,165.54) .. (191.69,197.84) .. controls (191.69,230.15) and (155.51,254.17) .. (122.48,254.17) .. controls (89.45,254.17) and (50.92,229.32) .. (50.92,197.01) .. controls (50.92,164.71) and (96.53,144) .. (96.53,112.52) ;
\draw    (98.1,197.01) .. controls (116.19,182.1) and (129.17,183.34) .. (146.08,196.18) ;
\draw    (91.81,192.87) .. controls (111.47,205.3) and (135.85,204.47) .. (150.01,193.7) ;
\draw    (96.53,112.52) .. controls (111.47,116.66) and (124.06,120.81) .. (140.57,111.69) ;
\draw    (96.53,112.52) .. controls (111.47,104.24) and (130.35,107.55) .. (140.57,111.69) ;
\draw [color={rgb, 255:red, 74; green, 144; blue, 226 }  ,draw opacity=1 ]   (122.09,202.4) .. controls (139.39,205.71) and (145.29,253.34) .. (127.99,254.17) ;
\draw [color={rgb, 255:red, 74; green, 144; blue, 226 }  ,draw opacity=1 ] [dash pattern={on 0.84pt off 2.51pt}]  (122.09,202.4) .. controls (105.84,201.84) and (112.65,252.65) .. (127.99,254.17) ;
\draw [color={rgb, 255:red, 208; green, 2; blue, 27 }  ,draw opacity=1 ]   (120.52,153.8) .. controls (203.36,152.14) and (188.15,231.11) .. (123.66,231.66) .. controls (59.17,232.22) and (49.74,153.25) .. (120.52,153.8) -- cycle ;
\draw    (265.86,106.45) .. controls (294.82,135.73) and (335.83,151.49) .. (382.48,106.45) ;
\draw    (252.44,141.23) .. controls (261,138.73) and (267.61,120.71) .. (265.86,106.45) ;
\draw    (258.66,227.82) .. controls (285.49,203.3) and (267.22,158.75) .. (252.44,141.23) ;
\draw    (279.27,251.35) .. controls (280.05,241.84) and (270.33,226.82) .. (258.66,227.82) ;
\draw    (279.27,251.35) .. controls (307.45,232.08) and (364.21,227.57) .. (395.9,251.35) ;
\draw    (395.9,251.35) .. controls (390.45,241.34) and (398.23,221.32) .. (409.89,221.82) ;
\draw    (409.89,221.82) .. controls (386.18,201.3) and (387.73,157.75) .. (403.67,135.23) ;
\draw    (403.67,135.23) .. controls (392.4,134.73) and (381.9,123.22) .. (382.48,106.45) ;
\draw    (252.44,141.23) .. controls (242.73,136.23) and (254.78,104.2) .. (265.86,106.45) ;
\draw    (279.27,251.35) .. controls (271.49,256.35) and (253.22,234.83) .. (258.66,227.82) ;
\draw    (382.48,106.45) .. controls (391.23,99.19) and (411.45,123.22) .. (403.67,135.23) ;
\draw    (395.9,251.35) .. controls (404.45,254.85) and (415.72,229.83) .. (409.89,221.82) ;
\draw [color={rgb, 255:red, 74; green, 144; blue, 226 }  ,draw opacity=1 ]   (326.83,134.75) .. controls (348.58,135.58) and (355.4,234.98) .. (330.76,235.53) ;
\draw [color={rgb, 255:red, 74; green, 144; blue, 226 }  ,draw opacity=1 ] [dash pattern={on 0.84pt off 2.51pt}]  (330.76,235.53) .. controls (310.84,236.08) and (310.31,134.47) .. (326.83,134.75) ;
\draw [color={rgb, 255:red, 208; green, 2; blue, 27 }  ,draw opacity=1 ]   (272.04,188.59) .. controls (273.61,210.13) and (391.58,196.87) .. (391.84,181.14) ;
\draw [color={rgb, 255:red, 208; green, 2; blue, 27 }  ,draw opacity=1 ] [dash pattern={on 0.84pt off 2.51pt}]  (272.04,188.59) .. controls (270.46,168.16) and (392.1,163.74) .. (391.84,181.14) ;

\draw (123,212.22) node [anchor=north west][inner sep=0.75pt]  [font=\tiny]  {$a_{i}$};
\draw (118.69,140.98) node [anchor=north west][inner sep=0.75pt]  [font=\tiny]  {$b_{i}$};
\draw (346.23,141.54) node [anchor=north west][inner sep=0.75pt]  [font=\tiny]  {$a_{i}$};
\draw (369.3,199.52) node [anchor=north west][inner sep=0.75pt]  [font=\tiny]  {$b_{i}$};

\end{tikzpicture}
\caption{Curves as in the proof of \autoref{proposition:zariski-dense-orbit}}	\label{fig:curve-choice}
	\end{figure}
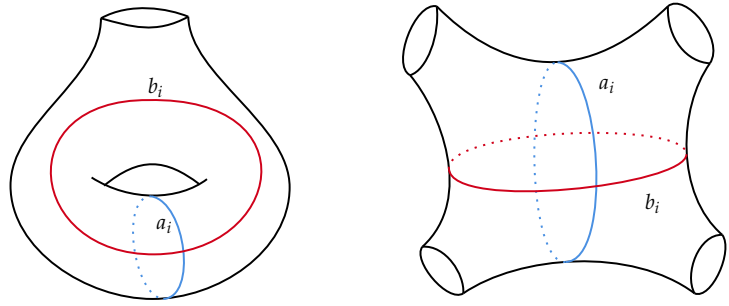
	
	We claim that $\tr_{\mc P}(\Gamma_{a_1,b_1, \dots, a_{3g-3+n}, b_{3g-3+n}}\cdot p)$ is Zariski dense in $\mb{A}^{3g-3+n}$. We will show by induction that $\tr_{a_1,\dots,a_m}(\Gamma_{a_1,b_1, \dots, a_{m}, b_{m}}\cdot p)$ is Zariski dense in $\mb A^{m}$ for $m=1,\dots, 3g-3+n$.
	
	Let us first deal with the base case $m=1$. Applying \autoref{lemma:zariski-dense-fiber} to $p$, $X_{g', n', k_1'}(\Sigma_1)$ and the pants decomposition of $\Sigma_1$ induced by $a_1$, we obtain that $\Gamma_{a_1} \cdot p$ is infinite. It follows from \autoref{lemma:dominant} that $\tr_{b_1}(\Gamma_{a_1}\cdot p)$ is infinite and, since $\tr_{b_1}(\Gamma_{a_1}\cdot p) \subset K$, we have that $\tr_{b_1}(\Gamma_{a_1}\cdot p) \setminus E$ is infinite. Pick any point $p' \in \Gamma_{a_1} \cdot p$ such that $\tr_{b_1}(p') \notin E$ and $p'$ lies in a perfect fiber for $\tr_{b_1}$ (there are only finitely many exceptions, due to \autoref{rem:P-good}). By \autoref{lemma:zariski-dense-fiber} applied to $p'$, $ X_{g', n', k_1'}(\Sigma_1)$ and the pants decomposition of $\Sigma_1$ induced by $b_1$, we obtain that $\Gamma_{b_1}\cdot p'$ is infinite, and so $\tr_{a_1}(\Gamma_{b_1}\cdot p')$ is infinite by \autoref{lemma:dominant}. Thus, we have showed that $\tr_{a_1}(\Gamma_{a_1, b_1} \cdot p)$ is infinite.
	
	We now deal with the induction step. Assume that $\tr_{a_1,\dots,a_m}(\Gamma_{a_1,b_1, \dots, a_{m}, b_{m}}\cdot p)$ is Zariski dense in $\mb A^{m}$. Let $\mc S$ be the set of $p' \in \Gamma_{a_1,b_1, \dots, a_{m}, b_{m}}\cdot p$ such that the restriction of $p'$ to $\Sigma_{m+1}$ is $\{a_{m+1}\}$-good. Pick any point $p' \in \mc S$. The same reasoning of the previous paragraph shows that $\tr_{a_{m+1}}(\Gamma_{a_{m+1}, b_{m+1}} \cdot p')$ is infinite. Since $a_{m+1}$ does not intersect $a_i$ and $b_i$ for all $i \neq m+1$, we have that $\tr_{a_{m+1}}(p')=\tr_{a_{m+1}}(p)$. It follows by \autoref{rem:P-good} that $\tr_{a_1,\dots,a_m}(\mc S)$ is Zariski dense in $\mb A^{m}$. Using again that $a_{m+1}$ and $b_{m+1}$ do not intersect $a_i$ and $b_i$ for all $i \neq m+1$, we also have that $\tr_{a_1,\dots,a_m}(\Gamma_{a_{m+1}, b_{m+1}} \cdot p')=\tr_{a_1,\dots,a_m}(p')$. Then $\tr_{a_1,\dots,a_{m+1}}(\Gamma_{a_{m+1}, b_{m+1}}\cdot \mc S)$ is Zariski dense in $\mb A^{m+1}$, and therefore also $\tr_{a_1,\dots,a_{m+1}}(\Gamma_{a_1,b_1, \dots, a_{m+1}, b_{m+1}}\cdot p)$ is Zariski dense in $\mb A^{m+1}$. 
%
\end{proof}

\section{Construction of integral representations}\label{section:construction_integral_representation}
In this section we will construct certain integral local systems on $\Sigma_{g,n}$ whose mapping class group orbit is Zariski-dense in the appropriate relative character variety---essentially, integral $\mathcal{P}$-good points in the terminology of \autoref{defn:p-good}. We will do so by gluing local systems on subsurfaces. We first
 introduce a class of integral matrices $\mathcal{M}_K$ such that local systems with peripheral monodromy in $\mathcal{M}_K$ are especially well-suited for such gluing.

\begin{definition} Let $K$ be a number field with ring of integers $\mc{O}_K$. 
	We define $\mc M_{K} \subset \SL_2(\mc O_K)$ to be the set of matrices
	\begin{equation*}
		\mc M_{K} = \left\{ \begin{pmatrix}
			a & u^{-1}(ad-1) \\
			u & d \\
		\end{pmatrix}   \, \Big | \, u \in \mc O_K^{\times}, \, a,d \in \mc O_K  \right\}.
	\end{equation*}
\end{definition}
\begin{remark}
	$\mc M_K$ is closed under inversion.
\end{remark}

\subsection{Surfaces of genus $0$}

\begin{lemma}\label{lemma:existence-integral-matrices-pant}
	Let $k_1, k_2 \in \mc O_K$ and let $A \in \mc M_K$. Then there exists a quadratic extension $L$ of $K$ and matrices $M_1, M_2 \in \SL_2(\mc O_L)$ such that $\tr M_i = k_i$ and $M_2=AM_1$. 
\end{lemma}
\begin{proof}
	Set 
	\begin{equation*}
		A=\begin{pmatrix}
			a & b \\
			c & d \\
		\end{pmatrix},
		M_1=\begin{pmatrix}
			x & y \\
			z & w \\
		\end{pmatrix}
	\end{equation*}
	so that our problem is equivalent to solving the following system
	\begin{equation*}
		\begin{cases}
			ax+bz+cy+dw=k_2\\
			x+w=k_1\\
			xw-yz=1\\
		\end{cases}
	\end{equation*}
	for $x,y,z,w \in \overline{\mb Z}$. This is equivalent to solving
	\begin{equation*}
		\begin{cases}
			(a-d)x+bz+cy=k_2-dk_1\\
			x^2-k_1x+yz+1=0\\
		\end{cases}
	\end{equation*}
	with $x,y,z \in \overline{\mb Z}$. Since $A \in \mc M_K$, we have $c \in \mc O_K^{\times}$, so the first equation gives
	\begin{equation*}
		y=c^{-1}(k_2-dk_1)-c^{-1}bz-c^{-1}(a-d)x.
	\end{equation*} 
	Substituting this expression for $y$ into the second equation, we get a monic quadratic equation in $x$ with coefficients in $\mc O_K[z]$. In particular, for any $z \in \mc O_K$, there exists a quadratic extension $L$ of $K$ and $x \in \mc O_L$ solving the equation. This concludes the proof.
\end{proof}

\begin{corollary}\label{corollary:existence-integral-representation-pant}
	Let $k_1, k_2 \in \mc O_K$ and let $A \in \mc M_K$ such that
	\begin{equation}\label{equation:numerical-condition-irreducibility}
		(\tr A)^2 + k_1^2+k_2^2-k_1k_2\tr A-2 \neq 2.
	\end{equation}
	Then there exists a quadratic extension $L$ of $K$ and a representation $\rho \colon \pi_1(\Sigma_{0,3}) \to \SL_2(\mc O_L)$ with monodromy $A$ along one puncture and trace $k_1$ and $k_2$ along the two other punctures, such that $\rho\otimes L$ is absolutely irreducible.
\end{corollary}
\begin{proof}
	The existence of a representation with the desired local monodromy and traces follows immediately from \autoref{lemma:existence-integral-matrices-pant}; irreducibility follows from \cite[Lemma $3.3$]{whang:ant}.
\end{proof}

\begin{lemma}\label{lemma:induction-step}
	Assume $\mc O_K^{\times}$ is infinite.
	Fix $k\in \mc O_K$ and $A \in \mc M_K$. Then there exists infinitely many pairs of matrices $(B,M) \in \mc M_K \times \SL_2(\mc O_K)$ such that $ABM=I$, $\tr M = k$, $\tr B \notin E$ and the traces of $A,B$ and $M$ satisfy \eqref{equation:numerical-condition-irreducibility}. Moreover, there are infinitely many values of $\tr B$ as $B$ varies among the above solutions.
\end{lemma}
\begin{proof}
	Since $M=(AB)^{-1}$, we need to find $B \in \mc M_K$ such that $\tr AB=k$. Assume that $A= \begin{pmatrix} a & u^{-1}(ad-1) \\ u & d  \end{pmatrix}$. Pick any $v \in \mc O_K^{\times}$ and let $B = \begin{pmatrix} x & v^{-1}(xw-1) \\ v & w  \end{pmatrix}$ with
	 \begin{equation*}
	 	\begin{cases}
	 		x=s^{-1}k-u^{-1}vd+s^{-1}(u^{-1}v+uv^{-1})\\
	 		w=u^{-1}v(s-a)\\
	 	\end{cases}
	 \end{equation*}
	for a suitable $s \in \mc O_K^{\times}$ that we will choose so that the required conditions on $\tr B$ will be satisfied. One can verify that the above choice of $B$ gives $\tr AB=k$. Using the above expressions for $x$ and $w$ one sees immediately that both $\tr B=x+w$ and $(\tr A)^2 + (\tr B)^2+k^2-k\tr A\tr B-4$ are non-constant rational functions in $K(s)$. In particular, for all but finitely many $s \in \mc O_K^{\times}$, the condition of \eqref{equation:numerical-condition-irreducibility} is satisfied and we have $\tr B \notin E$, proving the claim.
\end{proof}

\begin{proposition}\label{proposition:existence-representation-punctured-sphere}
	Let $K$ be a number field, $k=(k_1,\dots,k_{n}) \in (\mc O_K)^n$, $M \in \mc M_K$, $n \ge 3 $. Assume $\mc O_K^{\times}$ is infinite. There exists a pants decomposition $\mathcal{P}$ of $\Sigma=\Sigma_{0,n+1}$, a quadratic extension $L$ of $K$ and a $\mc P$-good integral representation $\rho \colon \pi_1(\Sigma) \to \SL_2(\mc O_L)$ such that
	\begin{itemize}
		\item the trace of the monodromy at the $i$-th puncture is $k_i$ for $i=1, \dots, n$,
		\item the monodromy at the $(n+1)$-st puncture is $M$.
	\end{itemize}
\end{proposition}
\begin{proof}
	Pick a simple loop $\gamma$ in $\Sigma$ that separates the $n+1$-st and the $n$-th puncture from the rest. This gives a decomposition $\Sigma_{0,n+1}=\Sigma_1 \cup \Sigma_{2}$ where $\Sigma_1$ is a pair of pants. Using \autoref{lemma:induction-step} we obtain an irreducible  $\SL_2(\mc O_K)$-representation on $\Sigma_1$ with monodromy of trace $k_n$ along the $n$-th puncture, monodromy $M$ along the $n+1$-th puncture and monodromy $M' \in \mc M_K$ along $\gamma$, with $\tr M' \notin E$. Moreover, when $n=3$, we choose $M'$ so that $\tr M'$, $k_1$ and $k_{2}$ satisfy the condition of \eqref{equation:numerical-condition-irreducibility}.

	We will argue by induction on $n$.
	
	If $n=3$, $\Sigma_2$ is a pair of pants, and using \autoref{corollary:existence-integral-representation-pant} we find a quadratic extension $L$ of $K$ and an irreducible $\mc O_L$-representation on $\Sigma_2$ with monodromy $M'$ along the puncture corresponding to $\gamma$. We may then glue along $\gamma$ the two representations we constructed on $\Sigma_1$ and $\Sigma_2$, thereby obtaining an $\mc O_L$ representation on $\Sigma_{0,4}$ satisfying the sought conditions.
	
	If $n \ge 4$, by the inductive hypothesis there exists a quadratic extension $L$ of $K$ and an $\mc O_L$-representation on $\Sigma_2$ satisfying the conditions of \autoref{proposition:existence-representation-punctured-sphere} with monodromy $M'$ along the puncture corresponding to $\gamma$. We may then glue along $\gamma$ the representations on $\Sigma_1$ and $\Sigma_2$, thereby obtaining a representation on $\Sigma_{0,n+1}$ satisfying the sought conditions.
\end{proof}

\subsection{Surfaces of positive genus}

\begin{lemma}\label{lemma:existence-representation-two-holed-torus}
	Assume $\mc O_K^{\times}$ is infinite. Let $M \in \mc M_K$. Then there exists a pants decomposition $\mc P$ of $\Sigma_{1,2}$ and a $\mc P$-good representation $\rho \colon \pi_1(\Sigma_{1,2}) \to \SL_2(\mc O_K)$ such that the monodromies along the first and second puncture are $M$ and $M^{-1}$, respectively.
\end{lemma}
\begin{proof}
	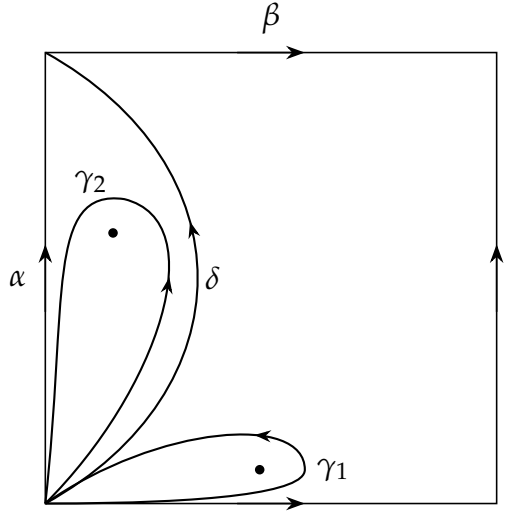
\begin{figure}[h]
		\begin{tikzpicture}[scale=1.5] 
			\def\L{4} 
			\coordinate (O) at (0,0); 
			\coordinate (TopLeft) at (0,\L); 
			
			\draw[semithick] (O) rectangle (\L,\L);
			
			\draw[-{Stealth[length=2.5mm, width=1.8mm]}, semithick] (0, \L*0.5 - 0.3) -- (0, \L*0.5 + 0.3);
			\node[left=0.10cm] at (0, \L/2) {$\alpha$};
			
			\draw[-{Stealth[length=2.5mm, width=1.8mm]}, semithick] (\L, \L*0.5 - 0.3) -- (\L, \L*0.5 + 0.3);
			
			\draw[-{Stealth[length=2.5mm, width=1.8mm]}, semithick] (\L*0.5 - 0.3, 0) -- (\L*0.5 + 0.3, 0);
			
			\draw[-{Stealth[length=2.5mm, width=1.8mm]}, semithick] (\L*0.5 - 0.3, \L) -- (\L*0.5 + 0.3, \L);
			\node[above=0.10cm] at (\L/2, \L) {$\beta$};
			
			\path[draw, thick, postaction={decorate, decoration={markings, mark=at position 0.6 with {\arrow{Stealth[length=2.2mm, width=1.6mm]}}}}]
			(O) .. controls (1.5, 0) and (2.3, 0.1) .. (2.3, 0.3)  
			.. controls (2.3, 0.8) and (1.0, 0.7) .. (O);        
			\node at (2.55, 0.3) {$\gamma_1$}; 
			\fill (1.9, 0.3) circle (1.2pt);   
			
			\path[draw, thick, postaction={decorate, decoration={markings, mark=at position 0.65 with {\arrow{Stealth[length=2.2mm, width=1.6mm, reversed]}}}}]
			(O) .. controls (0.2, 1.8) and (0.0, 2.8) .. (0.7, 2.7)   
			.. controls (1.5, 2.5) and (1.0, 1.0) .. (O);        
			\node at (0.4, 2.85) {$\gamma_2$};  
			\fill (0.6, 2.4) circle (1.2pt);    
			
			\path[draw, thick, postaction={decorate, decoration={markings, mark=at position 0.6 with {\arrow{Stealth[length=2.2mm, width=1.6mm]}}}}]
			(O) .. controls (1.8, 1.0) and (1.8, 3.0) .. (TopLeft); 
			\node[right=0.1cm] at (1.25, 2.0) {$\delta$}; 
			
		\end{tikzpicture}
		\caption{Generators of the fundamental group of a two-holed torus}
		\label{figure:two-holed-torus}
		\centering
	\end{figure}
	We consider generators $\alpha,\beta, \gamma_1,\gamma_2$ of $\pi_1(\Sigma_{1,2})$ as in \autoref{figure:two-holed-torus}, and we consider the pants decomposition $\mc P$ given by the paths $\alpha$ and $\delta$. The only relations in the fundamental group are $\alpha\beta\alpha^{-1}\beta^{-1}\gamma_1\gamma_2=1$ and $\delta= \gamma_2\alpha$.
	
	A representation $\rho$ as in the statement has to satisfy $\rho(\gamma_1)=M$ and $\rho(\gamma_2)=M^{-1}$, so that a $\mc P$-good $\rho$ is completely determined by $A\coloneqq \rho(\alpha)$ and $B \coloneqq \rho(\beta)$, with the following conditions:
	\begin{itemize}
		\item $A,B \in \SL_2(\mc O_K)$ and $ABA^{-1}B^{-1}=I$;
		\item $\tr A, \tr M^{-1}A \notin E$;
		\item The condition of \eqref{equation:numerical-condition-irreducibility} holds for $\tr A, \tr(M^{-1}A)$ and $\tr M$.
	\end{itemize}
	As an immediate consequence of the infinitude of $\mc O_K^{\times}$, we have that there exists infinitely many matrices $A=\begin{pmatrix} \lambda & t \\ 0 & \lambda^{-1}  \end{pmatrix}$ with $\lambda \in \mc O_K^{\times}$ and $t \in \mc O_K$ such that $\tr A, \tr M^{-1}A \notin E$ and the condition of \eqref{equation:numerical-condition-irreducibility} holds for $\tr A, \tr(M^{-1}A)$ and $\tr M$. This last condition follows from the fact that \eqref{equation:numerical-condition-irreducibility} corresponds to the representation being irreducible when restricted to the pair of pants, which means that $A, M, M^{-1}A$ have no common fixed point when viewed as linear automorphisms of $\mb P^1$: the fixed points of $A$ are $[1:0]$ and $[t:\lambda^{-1}-\lambda]$, so that suitably picking $\lambda$ and $t$ they are never fixed points of $M$ (we are using that the bottom left entry of $M$ is a unit, since $M \in \mc M_K$). 
	
	Finally, picking any $B \in \SL_2(\mc O_K)$ commuting with $A$, we see that the resulting representation satisfies the sought properties.
\end{proof}

To construct integral representations on a once-punctured torus, we introduce a subset of $\mc M_K$.
\begin{definition}
Let $\mc N_{K} \subset \SL_2(\mc O_K)$ be the set of matrices
\begin{equation*}
	\mc N_{K} = \left\{ \left[ \begin{pmatrix}
		a & u^{-1}(a-1) \\
		u & 1 \\
	\end{pmatrix} , \begin{pmatrix}
	v & 0 \\
	0 & v^{-1} \\
	\end{pmatrix} \right]   \, \Big | \, u,v,v-v^{-1} \in \mc O_K^{\times}, \, v \notin \mu_{\infty}, \, a \in \mc O_K  \right\}
\end{equation*}
\end{definition}
\begin{remark}
	$\mc N_K$ is closed under inversion.
\end{remark}
\begin{remark}\label{remark:special-unit}
	As long as $\mc O_K$ contains a unit $v \in \mc O_K^{\times}\setminus \mu_{\infty}$ such that $v-v^{-1}$ is also a unit, then $\mc N_K$ is infinite. For instance, it is sufficient that $ \sqrt{5} \in K$, so that $\frac{1+\sqrt{5}}{2} \in \mc O_K^{\times}$.
\end{remark}
\begin{remark}\label{remark:unit-scaling-nk}
	For $\lambda \in \mc O_K^{\times}$ we have
	\begin{equation*}
		\left[ \begin{pmatrix}
			a\lambda & u^{-1}(a-1)\lambda^{-1} \\
			u\lambda & \lambda^{-1} \\
		\end{pmatrix} , \begin{pmatrix}
			v & 0 \\
			0 & v^{-1} \\
		\end{pmatrix} \right]= \left[ \begin{pmatrix}
		a & u^{-1}(a-1) \\
		u & 1 \\
		\end{pmatrix} , \begin{pmatrix}
		v & 0 \\
		0 & v^{-1} \\
		\end{pmatrix} \right]
	\end{equation*}
	In particular, if $\mc O_K^{\times}$ is infinite, for any $M \in \mc N_K$ the set $\left\{ \tr A \, \Big | \, \left[A, \begin{pmatrix} 
		i & 0 \\
		0 & -i \\
	\end{pmatrix}\right]=M \right\}$ is infinite.
\end{remark}

\begin{remark}\label{remark:trace-of-commutator-matrix}
The matrices of $\mc N_K$ are of the form
\begin{equation*}
	M=\begin{pmatrix}
		a(1-v^{-2})+ v^{-2} & u^{-1}a(a-1)(1-v^2)\\[5pt]
		u(1-v^{-2}) & a(1-v^2)+v^2
	\end{pmatrix}
\end{equation*}
so that $\mc N_K \subset \mc M_K$.
We also have
\begin{equation}\label{equation:trace-formula}
	\tr M= -a(v-v^{-1})^2+v^2+v^{-2}
\end{equation} 
so that, since $-(v-v^{-1})^2$ is a unit, for any $k \in \mc O_K$ there exists a matrix $M \in \mc N_K$ with $\tr M=k$.
\end{remark}

The following result is a consequence of the definition of $\mc N_K$:
\begin{lemma}\label{lemma:existence-integral-representation-torus}
	For any $C \in \mc N_K$ there exists a pants decomposition $\mc P$ of $\Sigma_{1,1}$ and a $\mc P$-good representation $\rho \colon \pi_1(\Sigma_{1,1}) \to \SL_2(\mc O_K)$ with monodromy $C$ along the puncture.
\end{lemma}
\begin{proof}
	The fundamental group of $\Sigma_{1,1}$ is generated by paths $\alpha, \beta, \gamma$ with the condition $[\alpha, \beta]\gamma=1$, where $\gamma$ is the path going around the puncture. By \autoref{remark:unit-scaling-nk}, there exists $A,B \in\SL_2(\mc O_K)$ such that $[A,B]=C$, $\tr A \notin E$ and $(\tr A)^2 \neq 2+\tr C$. Let $\mc P = \{ \beta \}$ and consider the representation $\rho \colon \pi_1(\Sigma_{1,1}) \to \SL_2(\mc O_K)$ defined by $\rho(\alpha)=B$, $\rho(\beta)=A$ and $\rho(\gamma)=C$. We claim that $\rho$ is $\mc P$-good. If $\tr C = 2$, we are done by \autoref{definition:perfect-fiber} (since we are in the $(g,n,k)=(1,1,2)$ case) and the fact that $\tr A \notin E$. If $\tr C \neq 2$, we should check that \eqref{equation:numerical-condition-irreducibility} is satisfied by $\tr A, \tr A, \tr C$. Indeed \eqref{equation:numerical-condition-irreducibility} can be rewritten as
	\begin{equation*}
		((\tr A)^2 - (2+\tr C))(2-\tr C) \neq 0,
	\end{equation*}
	which is true by the hypotheses on $\tr A$ and $\tr C$.
\end{proof}

We may now construct integral $\mc P$-good representations on surfaces with one puncture:
\begin{proposition}\label{proposition:existence-representation-positive-genus}
	Let $K$ be a number field, $g \ge 1$ and $M \in \mc N_K$ such that $\tr M \notin E$. Assume that $\mc O_K^{\times}$ is infinite. There exists a pants decomposition $\mathcal{P}$ of $\Sigma=\Sigma_{g,1}$ and a $\mc P$-good representation $\rho \colon \pi_1(\Sigma) \to \SL_2(\mc O_K)$ with monodromy $M$ at the puncture.
\end{proposition}
\begin{proof}
	We proceed by induction on $g$. When $g=1$, this is \autoref{lemma:existence-integral-representation-torus}. If $g>1$, consider a separating path $\gamma$ cutting $\Sigma$ into two surfaces $\Sigma_1$ and $\Sigma_2$, where $\Sigma_1$ is a surface of type $(1,2)$ containing the puncture of $\Sigma$ (and with the other puncture along $\gamma$), and $\Sigma_2$ is of type $(g-1,1)$ with a puncture along $\gamma$. By \autoref{lemma:existence-representation-two-holed-torus} there exists an integral representation of $\pi_1(\Sigma_1)$ with monodromy $M$ along the puncture of $\Sigma$ and monodromy $M^{-1}$ along $\gamma$, and satisfying the conditions of \autoref{proposition:existence-representation-positive-genus}. Since $M^{-1} \in \mc N_K$, by the induction hypothesis there exists an integral representation on $\Sigma_2$ satisfying the conditions of \autoref{proposition:existence-representation-positive-genus}, and with monodromy $M^{-1}$ along $\gamma$. We may then glue the two representation and obtain the sought integral representation on $\Sigma$.
\end{proof}

\autoref{proposition:existence-representation-positive-genus} does not cover the case where the trace of the monodromy at the puncture belongs to $E$. In order to treat this case, we need the following:
\begin{lemma}\label{lemma:induction-step-family-N}
	Let $k\in \mc O_K \cap E$ and let $A \in \mc M_K$ with $\tr A \notin E$. Assume that there exists $v \in \mc O_K^{\times} \setminus \mu_{\infty}$ for which $v-v^{-1} \in \mc O_K^{\times}$. Then there exists a quadratic extension $L$ of $K$ and matrices $M \in \SL_2(\mc O_L)$ and $B \in \mc N_L$ such that $ABM=I$, $\tr M = k$, $\tr B \notin E$ and the traces of $A,B$ and $M$ satisfy the condition of \eqref{equation:numerical-condition-irreducibility}.
\end{lemma}
\begin{proof}
	We argue as in \autoref{lemma:induction-step}. Let $A= \begin{pmatrix} x & y \\ z & w  \end{pmatrix} \in \mc M_K$, so that $z \in \mc O_K^{\times}$. We will choose $L$ and $B \in \mc N_L$ of the form
	\begin{equation*}
		B=\begin{pmatrix}
			a(1-v^{-2})+ v^{-2} & u^{-1}a(a-1)(1-v^2)\\[5pt]
			u(1-v^{-2}) & a(1-v^2)+v^2
		\end{pmatrix}.
	\end{equation*}
	Our aim is to pick $a \in \mc O_L$ and $u \in \mc O_L^{\times}$ so that $\tr AB=k$, where $L$ is a suitable quadratic extension of $K$. The equation $\tr AB=k$ is
	\begin{equation}\label{equation:trace-of-product}
		x(a(1-v^{-2})+ v^{-2})+yu(1-v^{-2})+zu^{-1}a(a-1)(1-v^2)+w(a(1-v^2)+v^2)=k
	\end{equation}
	After dividing by $z(1-v^2)u^{-1}$ (which is a unit), this becomes a monic quadratic equation in $a$ with $\mc O_K$-integral coefficients. Therefore, for any $u \in \mc O_K^{\times}$, there exists a quadratic extension $L$ of $K$ and $a \in \mc O_L$ that solves \eqref{equation:trace-of-product}. Moreover, we claim that the coefficients of this quadratic equation are non-constant polynomials in $u$. Indeed if this was not the case, one would have that
	\begin{equation*}
		\begin{cases}
			x+w=xv^{-2}+wv^2\\
			x+w=k\\
		\end{cases}
	\end{equation*}
	which would imply that $\tr A = k \in E$, contradicting the hypotheses. In particular, as $u$ varies in $\mc O_K^{\times}$, the $a$'s solving \eqref{equation:trace-of-product} vary in an infinite set.
	
	In order to complete the proof, we just need to show that by choosing a suitable $u \in \mc O_K^{\times}$ we can ensure that $\tr B \notin E$ and that the traces of $A,B$ and $M$ satisfy \eqref{equation:numerical-condition-irreducibility}.
	
	We claim that only for finitely many such $u$'s we have $\tr B \in E$. Indeed, if $a$ solves \eqref{equation:trace-of-product}, then $[\mb Q(a) : \mb Q] \le 2[K : \mb Q] $ and, if $\tr B \in E$, then from \eqref{equation:trace-formula} one deduces that the Weil height $h(a)$ is bounded, so Northcott's theorem implies that there are only finitely many choices of $u$ with $\tr B\in E$.
	
	Similarly, substituting the traces of $A, B$ and $M$ in \eqref{equation:numerical-condition-irreducibility}, one finds a quadratic equation in $a$ with constant coefficients, so for all but (at most) two values of $a$, the condition of \eqref{equation:numerical-condition-irreducibility} is satisfied.
	
	The statement follows by picking any $u \in \mc O_K^{\times}$ such that the corresponding $a$ does not belong to the finite set where the two conditions are not satisfied.
\end{proof}
\begin{proposition}\label{proposition:doubly-punctured-torus}
	Fix $k \in \mc O_K \cap E$ and $M \in \mc M_K$ with $\tr M \notin E$. Assume that there exists $v \in \mc O_K^{\times} \setminus \mu_{\infty}$ for which $v-v^{-1} \in \mc O_K^{\times}$. Then there exists a quadratic extension $L$ of $K$, a pants decomposition $\mc P$ of $\Sigma_{1,2}$ and a $\mc P$-good representation $\rho \colon \pi_1(\Sigma_{1,2}) \to \SL_2(\mc O_L)$ such that the monodromy along one puncture is $M$ and the monodromy along the other puncture has trace $k$.
\end{proposition}
\begin{proof}
	Take a separating path $\gamma$ around the two punctures of $\Sigma_{1,2}$, so that $\gamma$ cuts $\Sigma_{1,2}$ into surfaces $\Sigma_1$ and $\Sigma_2$ of type $(0,3)$ and $(1,1)$, respectively, so that the puncture of $\Sigma_2$ corresponds to $\gamma$.
	
	By \autoref{lemma:induction-step-family-N}, there exists a quadratic extension $L$ of $K$ and an irreducible  $\SL_2(\mc O_L)$-local system $\rho_1$ on $\Sigma_1$, such that the monodromy of $\rho_1$ along the puncture corresponding to $\gamma$ is a matrix $N \in \mc N_L$ with $\tr N \notin E$, and at another puncture the monodromy is $M$, and at the last puncture it has trace $k$.
	
	By \autoref{lemma:existence-integral-representation-torus}, there exists a pants decomposition $\mc P'$ of $\Sigma_2$ and a $\mc P'$-good  $\SL_2(\mc O_L)$-local system $\rho_2$ on $\Sigma_2$ with monodromy $N$ along the puncture.
	
	Let $\mc P \coloneqq \{\gamma \} \cup \mc P'$ be a pants decomposition of $\Sigma_{1,2}$. Then,
	gluing $\rho_1$ and $\rho_2$, we obtain a $\mc P$-good  $\SL_2(\mc O_L)$-representation on $\Sigma_{1,2}$ satisfying the sought conditions. 
\end{proof}

\section{Proof of potential density}\label{sec:potential-density-sl2}
In this section we prove potential density of integral points on relative $\SL_2$-character varieties of surface groups.

We first construct integral $\mathcal{P}$-good representations in general:
\begin{proposition}\label{proposition:existence-good-representation}
	Let $\Sigma$ be a surface of type $(g,n)$ with $3g-3+n>0$, $K$ be number field and $k=(k_1,\dots,k_n) \in (\mc O_K)^n$. Assume that there exists $v \in \mc O_K^{\times} \setminus \mu_{\infty}$ for which $v-v^{-1} \in \mc O_K^{\times}$.
	Then there exists a quadratic extension $L$ of $K$, a pants decomposition $\mathcal{P}$ and a $\mc P$-good representation $\rho \colon \pi_1(\Sigma) \to \SL_2(\mc O_L)$ such that $[\rho] \in X_{g, n, k}(\mc O_L)$.
\end{proposition}
\begin{proof}
	Recall that the hypothesis on $K$ implies that $\mc N_K$ and $\mc O_K^{\times}$ are infinite. We will distinguish various cases depending on the values of $n$ and $g$.
	
	If $n=0$, then $g \ge 2$. Take a separating path $\gamma$ cutting $\Sigma$ into two surfaces $\Sigma_1$ and $\Sigma_2$, where $\Sigma_1$ is a surface of type $(1,1)$ with a puncture along $\gamma$ and $\Sigma_2$ is of type $(g-1,1)$ with a puncture along $\gamma$. Pick $M \in \mc N_K$ with $\tr M \notin E$ (this always exists by \autoref{remark:trace-of-commutator-matrix}); by \autoref{proposition:existence-representation-positive-genus} there exist good integral representations on $\Sigma_1$ and $\Sigma_2$ with monodromy $M$ along $\gamma$, so that the claim follows by gluing.
	
	If $n=1$ and $g=1$, by \autoref{remark:trace-of-commutator-matrix} there exists $M \in \mc N_K$ such that $\tr M=k$, and the claim follows from \autoref{lemma:existence-integral-representation-torus}. If $n=1$, $g>1$ and $k \notin E$, then the same reasoning works by using \autoref{proposition:existence-representation-positive-genus}. Instead, if $k \in E$, we pick $M \in \mc N$ with $\tr M \notin E$ and we consider a separating path $\gamma$ cutting $\Sigma$ into two surfaces $\Sigma_1$ and $\Sigma_2$, where $\Sigma_1$ is a surface of type $(1,2)$ containing the puncture of $\Sigma$ (and with the other puncture along $\gamma$), and $\Sigma_2$ is of type $(g-1,1)$ with a puncture along $\gamma$. Applying \autoref{proposition:doubly-punctured-torus} to $\Sigma_1$ and \autoref{proposition:existence-representation-positive-genus} to $\Sigma_2$, and possibly taking a quadratic extension of $K$, we find good integral representations with monodromy $M$ along $\gamma$, and the claim follows by gluing.
	
	If $n=2$, then $g>0$, so that we can take a separating path $\gamma$ cutting $\Sigma$ into two surfaces $\Sigma_1$ and $\Sigma_2$, where $\Sigma_1$ is a pair of pants containing the two punctures of $\Sigma$ (and with the other puncture along $\gamma$), and $\Sigma_2$ is of type $(g,1)$ with a puncture along $\gamma$. Pick any $M \in \mc N_K$ such that $\tr M$, $k_1$ and $k_2$ satisfy \eqref{equation:numerical-condition-irreducibility} and $\tr M \notin E$. The claim then follows applying \autoref{corollary:existence-integral-representation-pant} and \autoref{proposition:existence-representation-positive-genus} and gluing the constructed representations with monodromy $M$ along $\gamma$.
	
	If $n \ge 3$ and $g=0$, then $n \ge 4$. Pick $M \in \mc M_K$ such that $\tr M = k_n$: applying \autoref{proposition:existence-representation-punctured-sphere} we find a good integral representation satisfying the sought-after trace conditions at the punctures, possibly defined over a quadratic extension of $K$.
	
	If $n \ge 3$ and $g>0$, take a separating path $\gamma$ cutting $\Sigma$ into two surfaces $\Sigma_1$ and $\Sigma_2$, where $\Sigma_1$ is of type $(0,n+1)$ and contains all the $n$ punctures of $\Sigma$ (and with the other puncture along $\gamma$), and $\Sigma_2$ is of type $(g,1)$ with a puncture along $\gamma$. Pick any $M \in \mc N_K$ with $\tr M \notin E$. The claim follows by applying \autoref{proposition:existence-representation-punctured-sphere} and \autoref{proposition:existence-representation-positive-genus} and gluing the constructed representations with monodromy $M$ along $\gamma$.
\end{proof}
Combining \autoref{proposition:existence-good-representation} and \autoref{proposition:zariski-dense-orbit} we obtain:
\begin{theorem}\label{theorem:point-with-dense-orbit}
	Let $\Sigma$ be a surface of type $(g,n)$ with $3g-3+n>0$, $K$ be number field, and $k=(k_1,\dots,k_n) \in (\mc O_K)^n$. There exists $p \in X_{g,n,k}(\overline{\mb Z})$ such that $\Gamma_{g,n} \cdot p$ is Zariski dense in $X_{g,n,k, \overline{\mathbb{Q}}}$.
\end{theorem}
Before establishing \autoref{thm:main-theorem} for relative $\SL_2$-character varieties of surfaces, we separately treat the following special case:
\begin{lemma}\label{lemma:finite-orbit-torus}
	The surface $X \colon x^2+y^2+z^2-xyz-4=0$ contains a Zariski dense set of $\mb Z$-points.
\end{lemma}
\begin{proof}
	The automorphism $\tau_x(x,y,z)=(x,z,xz-y)$ of $X$ fixes $x$ and acts on $\begin{pmatrix}y\\z \end{pmatrix}$ as the matrix $\begin{pmatrix}0 & 1 \\-1 & x \\\end{pmatrix}$. Whenever $x \notin E$, we have that $\begin{pmatrix}0 & 1 \\-1 & x \\\end{pmatrix}^n-\begin{pmatrix}1 & 0 \\0 & 1\\\end{pmatrix}$ is invertible, which implies that $(x,0,0)$ is the only point of $\mathbb{A}^3$ with finite $\tau_x$-orbit and first coordinate equal to $x$. It follows that, for all $n \in \mb Z$, the point $(n,n,2)$ has a Zariski dense $\tau_x$-orbit in the curve $x=n$, proving the claim.
\end{proof}
We can finally prove the following:
\begin{theorem}\label{theorem:main-theorem-sl2} Let $\Sigma_{g,n}$ be an orientable surface of genus $g$ with $n$ punctures. Let $k\in \overline{\mathbb{Z}}^n$ be an $n$-tuple of algebraic integers.
	There exists a number field $K$ such that $X_{g,n,k}(\mc O_K)$ is Zariski dense.
\end{theorem}
\begin{proof}
	When $3g-3+n > 0$ the result follows from \autoref{theorem:point-with-dense-orbit}. If $3g-3+n \le  0$ and $g=0$, then the relative character variety is a point (see e.g.~\autoref{rem:P-good}), so the claim is trivial. If $3g-3+n \le  0$ and $g=1$, then $n=0$: by \autoref{proposition:degenerate-character-varieties} $X_{SL_2}(\Sigma_{1,0})$ is just the (full) character variety of a torus, that is $x^2+y^2+z^2-xyz-4=0$, so we conclude by \autoref{lemma:finite-orbit-torus}.
\end{proof}
\begin{remark}\label{rem:degree-4}
Unwinding the proof, one can check that if $K=\mathbb{Q}(k_1, \cdots, k_n)$ is the number field generated by the coordinates of $k\in \overline{\mathbb{Z}}^n$, then there exists a degree $4$ (in fact biquadratic) extension $L$ of $K$ such that $\mathscr{O}_L$-points are Zariski-dense in $X_{g,n,k, \overline{\mathbb{Q}}}$.
\end{remark}
\begin{remark}\label{rem:extension-necessary}
	Some extension is necessary; indeed, consider the variety $$x^2+y^2+z^2-xyz=3,$$ which is a relative character variety of $\Sigma_{1,1}$, by \autoref{subsec:1104}. Working mod $3$, we see that any $\mathbb{Z}$-point of this variety must have $x=y=z=0\bmod 3$. But then $x^2+y^2+z^2-xyz$ is divisible by $9$, and hence cannot equal $3$. Thus this relative character variety has no $\mathbb{Z}$-points.
\end{remark}

\section{ $\PGL_2$-character varieties}\label{sec:pgl2}
Let $\Sigma=\Sigma_{g,n}$ be an orientable surface of genus $g$ with $n$ punctures. The  $\PGL_2$-representation variety $\Hom(\pi_1(\Sigma),\PGL_2)$ is the affine scheme representing the functor $$A \mapsto \Hom(\pi_1(\Sigma),\PGL_2(A)).$$ The  $\PGL_2$-character variety $X_{\PGL_2}(\Sigma)$ is the (categorical) quotient $\Hom(\pi_1(\Sigma),\PGL_2)/\PGL_2$ under the action of  $\PGL_2$ by conjugation.

The regular function $\frac{\tr^2}{\det}$ descends from $\on{GL}_2$ to $\PGL_2$. Hence for each $a \in \pi_1(\Sigma)$ there is a regular function $f_a \coloneqq \frac{\tr_a^2}{\det_a} \colon \Hom(\pi_1(\Sigma), \PGL_2) \to \mb A^1 $ given by $\rho \mapsto \frac{\tr(\rho(a))^2}{\det(\rho(a))}$, which descend to a regular function $f_a \colon X_{\PGL_2}(\Sigma) \to \mb A^1$.
There is a natural morphism $f_{\partial \Sigma} \colon X_{\PGL_2}(\Sigma) \to \mb A^n$ sending $p \in X_{\PGL_2}(\Sigma)$ to the $n$-tuple of $f_a(p)$'s as $a$ varies along the $n$ boundary components of $\Sigma$. For $k \in \mb A^n$, we denote by $X_{\PGL_2, k}(\Sigma) \coloneqq f_{\partial \Sigma}^{-1}(k)$ the \emph{relative  $\PGL_2$-character variety} of $\Sigma$.

The goal of this section is to prove the following:
\begin{theorem}\label{theorem:pgl2-density}
	Let $\Sigma$ be a surface of type $(g,n)$ and $k \in \mb A^n(\Zbar)$. Then there exists a number field $K$ such that $X_{\PGL_2, k}(\Sigma)(\mc O_K)$ is Zariski dense.
\end{theorem}
We will do so by reducing the statement to the potential density of integral points on certain $\SL_2$-character varieties.

\subsection{Reduction to potential density on  $\SL_2$-character varieties}

Given $k=(k_1, \dots, k_n) \in \mb A^n(\Qbar)$, we let $\sqrt{k}_{+}=(\sqrt{k_1}, \dots, \sqrt{k_{n-1}}, \sqrt{k_n})$ and $\sqrt{k}_{-}=(\sqrt{k_1}, \dots, \sqrt{k_{n-1}}, -\sqrt{k_n})$ for a fixed choice of square roots. There are natural morphisms  $X_{\SL_2,\sqrt{k}_{+}}(\Sigma) \to X_{\PGL_2, k}(\Sigma)$ and $X_{\SL_2,\sqrt{k}_{-}}(\Sigma) \to X_{\PGL_2, k}(\Sigma)$. We have the following:
\begin{proposition}\label{proposition:pgl2-lift}
	Let $\Sigma$ be a surface of type $(g,n)$ with $n>0$ and let $k\in \mb A^n(\Qbar)$. Then each point of $X_{\PGL_2, k}(\Sigma)(\Qbar)$ lifts to either $X_{\SL_2,\sqrt{k}_{+}}(\Sigma)(\overline{\mb{Q}})$ or $X_{\SL_2,\sqrt{k}_{-}}(\Sigma)(\overline{\mb{Q}})$.
\end{proposition}
\begin{proof}
	Consider a standard presentation
	\begin{equation*}
		\pi_1(\Sigma)=\langle \alpha_1, \beta_1, \dots, \alpha_g, \beta_g, \gamma_1, \dots, \gamma_n \, \vert \, [\alpha_1,\beta_1] \cdots [\alpha_g,\beta_g] \gamma_1 \cdots \gamma_n=1 \rangle
	\end{equation*}
	of the fundamental group of $\Sigma$.
	Let $p \in  X_{\PGL_2, k}(\Sigma)(\Qbar)$ and let $\rho \in \Hom(\pi_1(\Sigma), \PGL_2(\Qbar))$ be a lift of $p$ to the representation variety. Choose $\widetilde{\rho}(\alpha_i), \widetilde{\rho}(\beta_i), \widetilde{\rho}(\gamma_i) \in \SL_2(\Qbar)$ lifting $\rho(\alpha_i), \rho(\beta_i), \rho(\gamma_i)$  such that $\tr \widetilde{\rho}(\gamma_i)=\sqrt{k_i}$ for $i=1, \dots, n$. In particular, we have that
	\begin{equation*}
		[\widetilde{\rho}(\alpha_1),\widetilde{\rho}(\beta_1)] \cdots [\widetilde{\rho}(\alpha_g),\widetilde{\rho}(\beta_g)] \widetilde{\rho}(\gamma_1) \cdots \widetilde{\rho}(\gamma_n)= \pm I,
	\end{equation*}
	where $I$ denotes the identity matrix. If the right hand side is $+I$, then $[\widetilde{\rho}]$ provides a lift of $p$ to $X_{\SL_2,\sqrt{k}_{+}}(\Sigma)(\overline{\mb{Q}})$. Otherwise, consider the  $\SL_2$-representation $\widetilde{\rho}_{-}$ which agrees with $\widetilde{\rho}$ on all generators of $\pi_1(\Sigma)$ except $\gamma_n$, where instead $\widetilde{\rho}_{-}(\gamma_n)\coloneqq -\widetilde{\rho}(\gamma_n)$. Then $[\widetilde{\rho}_{-}]$ provides a lift of $p$ to $X_{\SL_2,\sqrt{k}_{-}}(\Sigma)(\overline{\mb{Q}})$.
\end{proof}
\begin{corollary}\label{corollary:pgl2-lift}
	The map $X_{\SL_2,\sqrt{k}_{+}}(\Sigma) \sqcup X_{\SL_2,\sqrt{k}_{-}}(\Sigma) \to X_{\PGL_2, k}(\Sigma)$ is dominant.
\end{corollary}
In order to obtain a statement analogous to the above in the non-punctured case, we need to introduce the moduli space of $\SL_2$-local systems on a once-punctured surface with monodromy $-I$ at the puncture. Let $\Sigma$ be a surface of type $(g,1)$ and consider a standard presentation of its fundamental group
\begin{equation*}
	\pi_1(\Sigma)=\langle \alpha_1, \beta_1, \dots, \alpha_g, \beta_g, \gamma \, \vert \, [\alpha_1,\beta_1] \cdots [\alpha_g,\beta_g] \gamma =1 \rangle
\end{equation*}
where $\gamma$ is a loop around the puncture. Let $\Hom(\pi_1(\Sigma), \SL_2)_{-I}$ be the scheme representing the functor
\begin{equation*}
	A \mapsto \{\rho \in \Hom(\pi_1(\Sigma), \SL_2(A)) \, \vert \, \rho(\gamma)=-I \}.
\end{equation*} 
and let $X_{g,-I} \coloneqq \Hom(\pi_1(\Sigma), \SL_2)_{-I} /\SL_2$ be the categorical quotient under the action of  $\SL_2$ by  conjugation.
We have the following analogue of \autoref{proposition:pgl2-lift}:
\begin{proposition}\label{proposition:pgl2-lift-closed-surface}
	Let $\Sigma$ be a surface of type $(g,0)$ with $g\ge 1$. Then each point of $X_{\PGL_2}(\Sigma)(\Qbar)$ lifts to either $X_{\SL_2}(\Sigma)(\overline{\mb{Q}})$ or $X_{g,-I}(\overline{\mb{Q}})$.
\end{proposition}
\begin{proof}
	Analogous to \autoref{proposition:pgl2-lift}.
\end{proof}
\begin{corollary}\label{corollary:pgl2-lift-closed-surface}
	The map $X_{\SL_2}(\Sigma) \sqcup X_{g,-I} \to X_{\PGL_2}(\Sigma)$ is dominant.
\end{corollary}
The goal of the incoming sections is to prove the following:
\begin{theorem}\label{theorem:pgl2-density-minus-id}
	There exists a number field $K$ for which $X_{g,-I}(\mc O_K)$ is Zariski dense.
\end{theorem}
As a corollary of \autoref{theorem:pgl2-density-minus-id} and \autoref{theorem:main-theorem-sl2}, we can now obtain \autoref{theorem:pgl2-density}.
\begin{proof}[Proof of \autoref{theorem:pgl2-density}, assuming \autoref{theorem:pgl2-density-minus-id}]
	The statement follows from \autoref{corollary:pgl2-lift} and \autoref{theorem:main-theorem-sl2} if $n>0$, and from \autoref{corollary:pgl2-lift-closed-surface} and \autoref{theorem:pgl2-density-minus-id} if $n=0$.
\end{proof}

\subsection{Dynamics on $X_{g, -I} $}
Let $\Sigma$ be a surface of type $(g,1)$ with $g \ge 2$ and let $\mc P$ be a pants decomposition of $\Sigma$. The pair of pants containing the puncture of $\Sigma$ will have two more punctures corresponding to paths of $\mc P$, which we denote by $\gamma$ and $\delta$. Let us fix an enumeration $\mc P = \{a_1, \dots, a_{3g-3}, a_{3g-3+1} \}$, where $a_{3g-3}=\gamma$ and $ a_{3g-3+1}=\delta$. We consider the modified trace map
\begin{equation*}
	\tr_{\mc P, -I} \colon X_{g, -I} \to \mb A^{3g-3}
\end{equation*}
that sends a point of $X_{g, -I}$ to the traces along all paths of $\mc P$ \emph{except} $\delta$.

For $t \in \mb A^{3g-3}$, we denote $X^\mathcal{P}_{t,-I} \coloneqq \tr_{\mc P,-I}^{-1}(t)$.
Let
\begin{equation*}
	(-) \vert_{\Sigma \setminus \mc P} \colon X_{g,-I} \to X(\Sigma_1)\times \cdots \times X(\Sigma_{2g-2})
\end{equation*}
be the morphism induced by the immersion $\Sigma \setminus \mc P \to \Sigma$, where the product on the right hand side is taken over all the pair of pants associated to $\mc P$, with the exception of the pair of pants containing the puncture of $\Sigma$. Clearly $(-) \vert_{\Sigma \setminus \mc P}$ is constant along each fiber $X^\mathcal{P}_{t,-I}$.

\begin{definition}
	We say that $X^\mathcal{P}_{t,-I}$ is \emph{perfect} if
	\begin{itemize}
		\item for all $a_i \in \mc P$, we have $\tr_{a_i}(X^\mathcal{P}_{t,-I}) \neq \pm 2$ and
		\item $g=1$ or, for each $[\rho]$ in $X^\mathcal{P}_{t,-I}(\mb C)$, its restriction to each component of ${\Sigma \setminus \mc P}$ is irreducible, with the exception of the component containing the puncture of $\Sigma$.
	\end{itemize}
\end{definition}

Let $X^\mathcal{P}_{t,-I}$ be a perfect fiber. Notice that the action of $\Gamma_{\mc P}$ preserves $X^\mathcal{P}_{t,-I}$. Fix $(\lambda_1, \dots, \lambda_{3g-3}) \in (\Qbar^{\times})^{3g-3}$ such that $\lambda_i+\lambda_i^{-1}=t_i$. We denote by $T_{z_i} \colon \mb{G}_m^{3g-3} \to \mb{G}_m^{3g-3}$ the multiplication of the $i$-th coordinate by $\lambda_i$.
\begin{proposition}
	If $X^\mathcal{P}_{t,-I}$ be a perfect fiber, then there is a morphism
	\begin{equation*}
		F \colon \mb{G}_m^{3g-3} \to X^\mathcal{P}_{t,-I}
	\end{equation*}
	defined over $\Qbar$ satisfying the following:
	\begin{enumerate}
		\item at the level of $\Qbar$-points, $F$ is surjective with finite fibers,
		\item the action of $T_{z_i}$ on $\mb{G}_m^{3g-3}$ lifts the action of the Dehn twist $\tau_{a_i}$ on $X^\mathcal{P}_{t,-I}$.
	\end{enumerate}
\end{proposition}
\begin{proof}
	We refer the reader to the proof of \cite[Proposition $4.3$]{whang:ant}, as the argument is the same, with the only difference that in our case one obtains a morphism from $\mb{G}_m^{3g-3}$ to $X^\mathcal{P}_{t,-I}$, rather than from $\mb{G}_m^{3g-3+1}$ as in \cite[Proposition $4.3$]{whang:ant}. This is due to the fact that the restriction of a representation $\rho \in \Hom(\pi_1(\Sigma), \SL_2)_{-I}(\Qbar)$ to the pair of pants containing the puncture of $\Sigma$ is uniquely determined by $\rho(\gamma)$, since $\rho(\delta)=-\rho(\gamma)^{-1}$ as the monodromy at the puncture is $-I$.
\end{proof}

\begin{definition}
	Let $\mc P$ be a pants decomposition of $\Sigma$, $p \in X_{g,-I}$ and $t \coloneqq \tr_{\mc P,-I}(p)$. We say that $p$ is $\mc P$-good if $X^{\mc P}_{t,-I}$ is perfect and, for any $a \in \mc P \setminus \{ \gamma \cup \delta \}$, we have $\tr_a(X^\mathcal{P}_{t,-I}) \notin E$.
\end{definition}
Let $\Gamma_{-I}$ be the subgroup of the mapping class group of $\Sigma$ which is the identity on $\gamma$ and $\delta$. In particular $\Gamma_{-I}$ preserves the fibers of $\tr_{\gamma}$. We have the following:
\begin{proposition}\label{proposition:zariski-dense-orbit-minus-identity}
	Let $\mc P$ be a pants decomposition and let $p \in X_{g,-I}(\Qbar)$ be a $\mc P$-good point. Then $\Gamma_{-I} \cdot p$ is Zariski dense in $\tr_{\gamma}^{-1}(\tr_{\gamma}(p))$.
\end{proposition}
\begin{proof}
	The proof is the same as \autoref{proposition:zariski-dense-orbit}, except that we only use Dehn twists along paths that do not intersect $\gamma$ and $\delta$.
\end{proof}
In light of the previous proposition, in order to prove \autoref{theorem:pgl2-density-minus-id} it is sufficient to find a number field $K$ and a pants decomposition $\mc P$ of $\Sigma$ for which there exists infinitely many $\mc P$-good points $p \in X_{g,-I}(\mc O_K)$ with \emph{distinct} traces along $\gamma$. We will construct these points in the next section.

\subsection{Construction of an integral representation}
We introduce a class of matrices that will play the same role as the class $\mc N_K$ used in \autoref{section:construction_integral_representation}.
\begin{definition}
	Let $K$ be a finite extension of $\mb Q(i)$. We define $\mc L_{K} \subset \SL_2(\mc O_K)$ as the set of matrices
	\begin{align*}
		\mc L_{K} &= \left\{ \begin{pmatrix}
			ad+bc & 2ab \\
			2cd & ad+bc \\
		\end{pmatrix}  \, \Big | \, a, b, c, d \in \mc O_K, ad-bc=1, cd \neq 0  \right\}  \\
		&= \left\{ \left[ \begin{pmatrix}
			a & b \\
			c & d \\
		\end{pmatrix} , \begin{pmatrix}
			i & 0 \\
			0 & -i \\
		\end{pmatrix} \right]   \, \Big | \, a, b, c, d \in \mc O_K, ad-bc=1, cd \neq 0 \right\}
	\end{align*}
\end{definition}

\begin{remark}\label{remark:unit-scaling}
	For $\lambda \in \mc O_K^{\times}$ we have
	\begin{equation*}
		 \left[ \begin{pmatrix}
			a \lambda & b \lambda^{-1} \\
			c \lambda & d \lambda^{-1} \\
		\end{pmatrix} , \begin{pmatrix}
			i & 0 \\
			0 & -i \\
		\end{pmatrix} \right]= \left[ \begin{pmatrix}
		a & b \\
		c & d \\
		\end{pmatrix} , \begin{pmatrix}
		i & 0 \\
		0 & -i \\
		\end{pmatrix} \right].
	\end{equation*}
	In particular, if $\mc O_K^{\times}$ is infinite, for any $M \in \mc L_K$ the set $\left\{ \tr A \, \Big | \, \left[A, \begin{pmatrix} 
	i & 0 \\
	0 & -i \\
	\end{pmatrix}\right]=M \right\}$ is infinite.
\end{remark}

\begin{remark}\label{rem:minus-inverse}
	If $M \in \mc L_K$ then $M^{-1} \in \mc L_K$ and $-M \in \mc L_K$. Indeed
	\begin{equation*}
		\left[ \begin{pmatrix}
			a & b \\
			c & d \\
		\end{pmatrix} , \begin{pmatrix}
			i & 0 \\
			0 & -i \\
		\end{pmatrix} \right]^{-1} = \left[ \begin{pmatrix}
		-a & b \\
		c & -d \\
		\end{pmatrix} , \begin{pmatrix}
		i & 0 \\
		0 & -i \\
		\end{pmatrix} \right] 
	\end{equation*} 
	and
	\begin{equation*}
		-\left[ \begin{pmatrix}
			a & b \\
			c & d \\
		\end{pmatrix} , \begin{pmatrix}
			i & 0 \\
			0 & -i \\
		\end{pmatrix} \right] = \left[ \begin{pmatrix}
			-b & a \\
			-d & c \\
		\end{pmatrix} , \begin{pmatrix}
			i & 0 \\
			0 & -i \\
		\end{pmatrix} \right] 
	\end{equation*} 
\end{remark}
The class of matrices $\mc L_K$ satisfies similar properties as $\mc N_K$.
Specifically, we have the following:
\begin{lemma}\label{lemma:existence-representation-two-holed-torus-minus-id}
	Assume $\mc O_K^{\times}$ is infinite. Let $M \in \mc L_K$. Then there exists a pants decomposition $\mc P$ of $\Sigma_{1,2}$ and a $\mc P$-good representation $\rho \colon \pi_1(\Sigma_{1,2}) \to \SL_2(\mc O_K)$ such that the monodromies along the first and second puncture are $M$ and $M^{-1}$, respectively.
\end{lemma}
\begin{proof}
	The proof is the same as \autoref{lemma:existence-representation-two-holed-torus}, and crucially uses that the bottom left entry of $M$ is non-zero.
\end{proof}

\begin{lemma}\label{lemma:existence-integral-representation-torus-minus-id}
	Assume $\mc O_K^{\times}$ is infinite. For any $M \in \mc L_K$ there exists a pants decomposition $\mc P$ of $\Sigma_{1,1}$ and a $\mc P$-good representation $\rho \colon \pi_1(\Sigma_{1,1}) \to \SL_2(\mc O_K)$ with monodromy $M$ along the puncture.
\end{lemma}
\begin{proof}
	This is an immediate consequence of \autoref{remark:unit-scaling}.
\end{proof}

\begin{proposition}\label{proposition:existence-representation-positive-genus-minus-id}
	Let $K$ be a number field, $g \ge 1$ and $M \in \mc L_K$ such that $\tr M \notin E$. Assume that $\mc O_K^{\times}$ is infinite. There exists a pants decomposition $\mathcal{P}$ of $\Sigma=\Sigma_{g,1}$ and a $\mc P$-good representation $\rho \colon \pi_1(\Sigma) \to \SL_2(\mc O_K)$ with monodromy $M$ at the puncture.
\end{proposition}
\begin{proof}
	The proof is analogous to \autoref{proposition:existence-representation-positive-genus} and crucially uses that $\mc L_K$ is closed under inversion.
\end{proof}

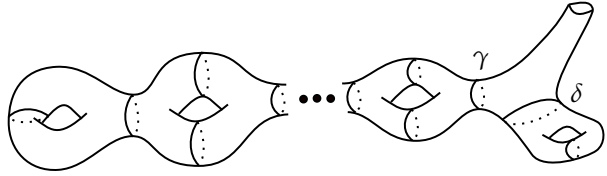
\begin{figure}[h]

\tikzset{every picture/.style={line width=0.75pt}} 

\begin{tikzpicture}[x=0.75pt,y=0.75pt,yscale=-1,xscale=1]

\draw    (257,161.8) .. controls (234,162.8) and (236,187.8) .. (212,187.8) .. controls (188,187.8) and (190.38,172.43) .. (178,172.8) .. controls (165.63,173.18) and (156,188.8) .. (141,189.8) .. controls (126,190.8) and (116,179.8) .. (116,165.8) .. controls (116,151.8) and (122,138.8) .. (140,137.8) .. controls (158,136.8) and (167.75,152.8) .. (180,151.8) .. controls (192.25,150.8) and (187,131.8) .. (212,130.8) .. controls (237,129.8) and (231,146.8) .. (256,146.8) ;
\draw    (132.95,166) .. controls (144.72,152) and (145.26,157.2) .. (152.86,164) ;
\draw    (128.6,163.6) .. controls (137.11,169.4) and (140.55,174.8) .. (155.4,161.2) ;
\draw    (201.8,161.6) .. controls (214.8,147.6) and (215.4,152.8) .. (223.8,159.6) ;
\draw    (197,159.2) .. controls (206.4,165) and (210.2,170.4) .. (226.6,156.8) ;
\draw  [dash pattern={on 0.84pt off 2.51pt}]  (251.67,147.14) .. controls (254.67,150.73) and (257.07,158.73) .. (252.67,162.14) ;
\draw    (251.67,147.14) .. controls (247.07,151.93) and (247.47,158.73) .. (252.67,162.14) ;
\draw  [fill={rgb, 255:red, 14; green, 0; blue, 0 }  ,fill opacity=1 ] (263,154.1) .. controls (263,153.16) and (263.76,152.4) .. (264.7,152.4) .. controls (265.64,152.4) and (266.4,153.16) .. (266.4,154.1) .. controls (266.4,155.04) and (265.64,155.8) .. (264.7,155.8) .. controls (263.76,155.8) and (263,155.04) .. (263,154.1) -- cycle ;
\draw  [fill={rgb, 255:red, 14; green, 0; blue, 0 }  ,fill opacity=1 ] (269.8,153.7) .. controls (269.8,152.76) and (270.56,152) .. (271.5,152) .. controls (272.44,152) and (273.2,152.76) .. (273.2,153.7) .. controls (273.2,154.64) and (272.44,155.4) .. (271.5,155.4) .. controls (270.56,155.4) and (269.8,154.64) .. (269.8,153.7) -- cycle ;
\draw  [fill={rgb, 255:red, 14; green, 0; blue, 0 }  ,fill opacity=1 ] (276.71,153.81) .. controls (276.71,152.88) and (277.48,152.11) .. (278.41,152.11) .. controls (279.35,152.11) and (280.11,152.88) .. (280.11,153.81) .. controls (280.11,154.75) and (279.35,155.51) .. (278.41,155.51) .. controls (277.48,155.51) and (276.71,154.75) .. (276.71,153.81) -- cycle ;
\draw    (284.17,146.17) .. controls (296.83,146.67) and (302.5,133.83) .. (319.5,133.67) .. controls (336.5,133.5) and (339.83,144.67) .. (350.5,144.67) .. controls (361.17,144.67) and (371.5,138.67) .. (379.5,130.33) .. controls (387.5,122) and (392.83,117) .. (398.5,106.33) ;
\draw    (287,160.64) .. controls (299.67,161.14) and (304.67,175.3) .. (321.67,175.14) .. controls (338.67,174.97) and (342.67,159.14) .. (353.33,159.14) .. controls (364,159.14) and (373.5,173.33) .. (379.83,182) .. controls (386.17,190.67) and (407.67,183.47) .. (412.33,180.14) .. controls (417,176.8) and (416.83,169.33) .. (413.83,166.33) .. controls (410.83,163.33) and (396.75,161.13) .. (392.83,153.67) .. controls (388.92,146.21) and (408.67,117.8) .. (410.67,109.14) ;
\draw  [dash pattern={on 0.84pt off 2.51pt}]  (352.33,144.14) .. controls (355.33,147.73) and (357.73,155.73) .. (353.33,159.14) ;
\draw    (352.33,144.14) .. controls (347.73,148.93) and (348.13,155.73) .. (353.33,159.14) ;
\draw    (308.8,155.6) .. controls (321.8,141.6) and (322.4,146.8) .. (330.8,153.6) ;
\draw    (304,153.2) .. controls (313.4,159) and (317.2,164.4) .. (333.6,150.8) ;
\draw    (388.75,173.62) .. controls (398.46,164.72) and (398.9,168.03) .. (405.18,172.35) ;
\draw    (385.17,172.1) .. controls (392.18,175.78) and (395.02,179.22) .. (407.27,170.57) ;
\draw    (364.83,164.67) .. controls (368.5,161.17) and (385.5,150.67) .. (394.17,155.67) ;
\draw  [dash pattern={on 0.84pt off 2.51pt}]  (394.17,155.67) .. controls (397.17,159.26) and (372.5,171.67) .. (364.83,164.67) ;
\draw    (398.5,106.33) .. controls (403.13,105.62) and (410.02,103.61) .. (410.67,109.14) ;
\draw    (398.63,106.2) .. controls (400.36,112.61) and (405.07,111.84) .. (410.67,109.14) ;
\draw  [dash pattern={on 0.84pt off 2.51pt}]  (290,145.8) .. controls (293,149.4) and (295.4,157.4) .. (291,160.8) ;
\draw    (290,145.8) .. controls (285.4,150.6) and (285.8,157.4) .. (291,160.8) ;
\draw  [dash pattern={on 0.84pt off 2.51pt}]  (178.29,152) .. controls (181.15,157.18) and (183.24,168.48) .. (178.71,173) ;
\draw    (178.29,152) .. controls (173.51,158.45) and (173.64,167.96) .. (178.71,173) ;
\draw  [dash pattern={on 0.84pt off 2.51pt}]  (213.55,130.33) .. controls (216.62,135.85) and (218.87,147.86) .. (214,152.67) ;
\draw    (213.55,130.33) .. controls (208.41,137.19) and (208.56,147.31) .. (214,152.67) ;
\draw  [dash pattern={on 0.84pt off 2.51pt}]  (211.29,165.67) .. controls (214.15,171.01) and (216.24,182.67) .. (211.71,187.33) ;
\draw    (211.29,165.67) .. controls (206.51,172.32) and (206.64,182.14) .. (211.71,187.33) ;
\draw  [dash pattern={on 0.84pt off 2.51pt}]  (319.29,159.33) .. controls (322.15,163.2) and (324.24,171.63) .. (319.71,175) ;
\draw    (319.29,159.33) .. controls (314.51,164.14) and (314.64,171.24) .. (319.71,175) ;
\draw  [dash pattern={on 0.84pt off 2.51pt}]  (320.62,134) .. controls (323.48,137.13) and (325.57,143.94) .. (321.04,146.67) ;
\draw    (320.62,134) .. controls (315.84,137.89) and (315.98,143.63) .. (321.04,146.67) ;
\draw  [dash pattern={on 0.84pt off 2.51pt}]  (136.17,162.73) .. controls (131.16,165.58) and (120.26,167.65) .. (115.91,163.12) ;
\draw    (136.17,162.73) .. controls (129.95,157.94) and (120.77,158.06) .. (115.91,163.12) ;
\draw  [dash pattern={on 0.84pt off 2.51pt}]  (400.62,174.33) .. controls (403.48,176.88) and (405.57,182.44) .. (401.04,184.67) ;
\draw    (400.62,174.33) .. controls (395.84,177.51) and (395.98,182.19) .. (401.04,184.67) ;

\draw (348.5,128.23) node [anchor=north west][inner sep=0.75pt]  [font=\footnotesize]  {$\gamma $};
\draw (397.83,144.57) node [anchor=north west][inner sep=0.75pt]  [font=\footnotesize]  {$\delta $};

\end{tikzpicture}

	\caption{A pants decomposition as in the proof of \autoref{proposition:existence-good-representation-minus-id}}
	\label{figure:pgl2construction}
	\centering
\end{figure}

We may now prove the main result of this section:
\begin{proposition}\label{proposition:existence-good-representation-minus-id}
	Let $K$ be a number field such that $\mc O_K^{\times}$ is infinite.
	Let $g \ge 2$, $\Sigma$ a surface of type $(g,1)$, $\mc P$ the pants decomposition of \autoref{figure:pgl2construction} and $M \in \mc L_K$ such that $\tr M \notin E$. There exists a $\mc P$-good  $\SL_2(\mc O_K)$-representation with monodromy $-I$ along the puncture and monodromy $M$ along $\gamma$.
\end{proposition}
\begin{proof}
	Consider separating paths $\gamma$ and $\delta$ as in \autoref{figure:pgl2construction} that cut $\Sigma$ into the following three surfaces:
	\begin{enumerate}
		\item $\Sigma_1$ of type $(g-1,1)$ with a puncture along $\gamma$,
		\item $\Sigma_2$ of type $(0,3)$, where the punctures correspond to $\gamma$, $\delta$ and the puncture of $\Sigma$,
		\item $\Sigma_3$ of type $(1,1)$, with a puncture along $\delta$.
	\end{enumerate}
	By \autoref{proposition:existence-representation-positive-genus-minus-id}, there exists a pants decomposition $\mc P_1$ of $\Sigma_1$ and a $\mc P_1$-good  $\SL_2(\mc O_K)$-representation $\rho_1$ on $\Sigma_1$ with monodromy $M$ along $\gamma$.
	Since $-M^{-1} \in \mc L_K$ (see \autoref{rem:minus-inverse}), by \autoref{lemma:existence-integral-representation-torus-minus-id} there exists a pants decomposition $\mc P_3$ of $\Sigma_3$ and a $\mc P_3$-good  $\SL_2(\mc O_K)$-representation $\rho_3$ on $\Sigma_3$ with monodromy $-M^{-1}$ along $\delta$.
	Finally, consider the  $\SL_2(\mc O_K)$-representation $\rho_2$ on $\Sigma_2$ with monodromy $M$ along $\gamma$, $-M^{-1}$ along $\delta$ and $-I$ at the puncture.
	
	Let $\mc P=\mc P_1 \cup \{ \gamma\} \cup \{\delta\} \cup \mc P_3$.
	Since $\rho_1,\rho_2$ and $\rho_3$ agree on $\gamma$ and $\delta$, we may glue them and find the sought $\mc P$-good  $\SL_2(\mc O_K)$-representation $\rho$ on $\Sigma$ with monodromy $-I$ along the puncture and monodromy $M$ along $\gamma$.	
\end{proof}

\begin{proof}[Proof of \autoref{theorem:pgl2-density-minus-id}]
	If $g=1$ the claim is trivial since $X_{1,-I}$ is a single point, as it corresponds to the point $(0,0,0)$ of the Markov surface $x^2+y^2+z^2-xyz=0$ (see \autoref{proposition:degenerate-character-varieties}). If $g>1$, enlarge $K$ so that $\mc O_K^{\times}$ is infinite and consider a subset $\{ M_n\}_{n \ge 1} \subset \mc L_K$ such that $\tr M_n \notin E$ for all $n$ and the set $\{\tr M_n \}_{n\ge 1}$ is infinite. By \autoref{proposition:existence-good-representation-minus-id} and \autoref{proposition:zariski-dense-orbit-minus-identity}, for each $n$ the fiber $\tr_{\gamma}^{-1}(\tr M_n)$ of $\tr_{\gamma} \colon X_{g,-I} \to \mb A^1 $ contains a Zariski dense set of $\mc O_K$-points. The claim follows from the infinitude of $\{\tr M_n \}_{n\ge 1}$.
\end{proof}

\section{Proof of \autoref{thm:main-theorem}}\label{sec:corlette-simpson}
We now begin preparations for the proof of \autoref{thm:main-theorem}, on potential density of integral points in relative character varieties of smooth quasi-projective varieties $Y$ with $\dim Y >1$. We will require a few preparatory lemmas. 
\subsection{Relative character varieties and morphisms}
let $Y$ be a smooth complex variety equipped with a projective simple normal crossings compactification $\overline{Y}$, with boundary divisor $D=\cup_{i=1}^n D_i$. Let $G$ be $\SL_{2, \mathbb{Z}}$ or $\PGL_{2, \mathbb{Z}}$, $K$ a number field, and $\underline{C}=(C_1, \cdots, C_n)\in (G/_{\text{ad}}G)(\mathscr{O}_K)^n$. Recall that we are studying $X_{G, \underline{C}}(Y)$, the relative character variety parametrizing representations with fixed traces along $D_i$, defined as in the introduction.

\begin{remark}
A priori $X_{G, \underline{C}}(Y)$ depends on the compactification $\overline{Y}$, but in fact this is not the case---further blow-ups add additional components $D_j'$ to $D$ but their boundary data is already determined by $\underline{C}$ (as a small loop around the exceptional divisor is a product of \emph{commuting} loops about the strict transforms of the components of $D$ it meets). Given two simple normal crossings compactifications of $Y$, we may thus dominate them by a third to compare relative character varieties. We leave verifying the details to the reader; we will in what follows freely replace $\overline{Y}$ with a blowup.
\end{remark}

\begin{lemma}\label{lem:integral-conjugacy-class}
Let $Z$ be an orbicurve, i.e.~a smooth Deligne-Mumford curve containing a scheme as a dense open subset. With notation as above, let $[\rho]\in X_{G,\underline{C}}(Y)$ be a point and $f: Y\to Z$ be a morphism with connected fibers so that $\rho$ factors through the induced map $\pi_1(Y)\to \pi_1(Z)$ (i.e.~$\rho$ ``factors through an orbicurve" in the language of \cite{corlette2008classification}). Then there exists $\underline{C}'\in (G/_{\text{ad}}G)(\mathscr{O}_K)^m$, for appropriate $m$, so that $[\rho]$ is in the image of the induced map $$f^*: X_{G, \underline{C}'}(Z)\to X_{G, \underline{C}}(Y).$$
\end{lemma}
\begin{proof}
	That there exists some $\underline{C}'$ such that $[\rho]$ is in the image of  $f^*: X_{G, \underline{C}'}(Z)\to X_{G, \underline{C}}(Y)$ is clear by the assumption that $\rho$ factors through an orbicurve; all that needs to be checked is that we may take $\underline{C}'\in (G/_{\text{ad}}G)(\mathscr{O}_K)^m$, i.e.~that it is integral. After replacing $\overline{Y}$ by a blowup, we may assume $f$ extends to a map $$\overline{f}: \overline{Y}\to \overline{Z},$$ where $\overline{Z}$ is a smooth proper orbicurve containing $Z$. Let $E=\overline{Z}\setminus Z$. 
	
	Now each component $D_i$ of $D=\overline{Y}\setminus Y$ either 
	\begin{enumerate}
	\item 	dominates $\overline{Z}$,
	\item maps to a point of $Z$, or
	\item maps to a point of $E$.
	\end{enumerate}
In the first two cases, we have that $C_i$ is the class of the identity in $(G/_{\text{ad}}G)(\mathscr{O}_K)^m$.  In the last case, a small loop around $\overline{f}(D_i)$ has some multiple lifting to a small loop around $D_i$. Hence we may take $C'_i$ to be such that some integer power of it is in $C_i$, whence it is  in $(G/_{\text{ad}}G)(\mathscr{O}_K)$ as desired.
\end{proof}

\subsection{Non-Zariski dense representations} We start by recalling the classification of maximal-Zariski closed subgroups of $\SL_{2, \mathbb{C}}$, resp.~$PGL_{2,\mathbb{C}}$. Any such is either finite, a Borel (conjugate to the subgroup of matrices of the form $$\begin{pmatrix} a & b\\0 & c\end{pmatrix},$$ where the $a,b,c\in \mathbb{C}$ and $ac=1$) or the normalizer of a maximal torus (conjugate to the subgroup of matrices of the form $$\begin{pmatrix} a & 0 \\ 0 & b\end{pmatrix} \text{ or } \begin{pmatrix} 0 & c \\ d & 0\end{pmatrix}$$ with $ab=1$, resp.~$cd=-1$). We first observe that integral points are potentially dense in the subset of $X_{G, \underline{C}}(Y)$ consisting of representations that may be conjugated into one of these maximal Zariski-closed subgroups.
\begin{lemma}\label{lemma:non-dense-reps}
Let $Y$ be a smooth complex variety equipped with a smooth projective simple normal crossings compactification $\overline{Y}$, with $D=\overline{Y}\setminus Y$, and $D=\cup_{i=1}^n D_i$ the irreducible components of $D$. Let $G=\PGL_{2, \mathbb{Z}}$ or $G=\SL_{2, \mathbb{Z}}$ and fix a tuple of points $\underline{C}=(C_1, \cdots, C_n)\in (G/_{\text{ad}} G)^n(\mathscr{O}_K)$ for some number field $K$. Let $Z\subset X_{G, \underline{C}}(Y)_{\overline{\mathbb{Q}}}$ be the closed subscheme consisting of representations whose image is not Zariski-dense. There exists a number field $K'\supset K$ such that $Z(\mathscr{O}_{K'})$ is Zariski-dense in $Z$.
\end{lemma}
\begin{proof}
	Let $Z_i$ be an irreducible component of $Z$, and $\eta_i$ the generic point of $Z_i$. Let $\overline{\kappa(\eta_i)}$ be an algebraic closure of the residue field of $\eta_i$, and $\rho_i: \pi_1(Y)\to G(\overline{\kappa(\eta_i)})$ the corresponding representation. There are three cases:
	\begin{enumerate}
		\item $\rho_i$ has finite image: in this case, $\rho_i$ is already defined over the ring of integers of some number field and hence is Zariski-dense in $Z_i$, which is a point.
		\item $\rho_i$ has image contained in a Borel. In this case the same is true for all $[\rho]$ in $Z_i$. Each such $\rho$ is $S$-equivalent to a representation factoring through a maximal torus $T$ of $G$, i.e.~it corresponds to the same point of the character variety as such a representation.
			
		Any representation of $\pi_1(Y)$ into $T$ factors through $\pi_1(Y)^{\text{ab}}$, a finitely-generated Abelian group, say $\mathbb{Z}^r\oplus F$ with $F$ finite. The set of such is isomorphic to $T^r\times F^\vee$, where $F^\vee=\text{Hom}(F, T)$. Let $X_{T, \underline{C}}(Y)$ be the preimage of $X_{G, \underline{C}}(Y)$ in $X_{T}(Y)$, under the map induced by $T\hookrightarrow G$.
		
		The local monodromy condition is affine-linear on $T^r\times F^\vee$, whence $X_{T, \underline{C}}(Y)$ is a component of a torsor for a torus $T'$ times a finite group $W$. But any such admits a potentially dense set of integral points---simply enlarge $K$ so it splits this torsor, and so that integral points are dense in $T'\times W$; to obtain potential density of integral points, adjoin enough roots of unity to split $W$, enlarge $K$ to split $T$, and further enlarge it to have infinite unit group. 
		
		\item $\rho_i$ has image contained in the normalizer of a maximal torus, $D$. Note that $D$ has identity component a maximal torus $T$, and $D/T$ has order $2$, acting on $T$ by inversion; $D$ evidently splits as $T\rtimes \{\pm 1\}$. Let $$\psi: \pi_1(Y)\overset{\rho_i}{\longrightarrow} D(\overline{\kappa(\eta_i)})\to D/T(\overline{\kappa(\eta_i)})=\{\pm 1\}$$ be the composition of $\rho_i$ with the natural quotient map $D\to D/T$. Let $\tilde Y$ be the \'etale double cover of $Y$ corresponding to the kernel of $\psi$. For any $[\rho]$ in $Z_i$, $\rho|_{\pi_1(\tilde Y)}$ factors through $T$; denote by $\tilde\rho$ this (one-dimensional) representation $\pi_1(\tilde Y)\to T$.
		
		Now for $G=\SL_2$, we have $\rho=\text{Ind}_{\pi_1(\tilde Y)}^{\pi_1(Y)}\tilde \rho$. Hence potential density of integral points follows as in point (2) above---it suffices to prove density for representations $\pi_1(\tilde Y)\to T$, which we did above (replacing $Y$ with $\tilde Y$). For $G=\PGL_2$, $\rho=\mathbb{P}\text{Ind}_{\pi_1(\tilde Y)}^{\pi_1(Y)}\tilde \rho$, and the same argument suffices.
	\end{enumerate}
\end{proof}
\subsection{Lifting from $\PGL_2$ to $\SL_2$}
\begin{lemma}\label{lemma:lifting-reps}
	Let $\Gamma$ be a finitely-generated group, $K$ a number field, and $K'$ the compositum of all degree $2$ extensions of $K$ ramified only over the prime $2$. Let $\rho: \Gamma\to\PGL_2(\mathscr{O}_K)$ be a representation. Suppose there exists $\rho': \Gamma\to\SL_2(\mathbb{C})$ such that the composition $$\Gamma\overset{\rho'}{\longrightarrow}\SL_2(\mathbb{C})\to\PGL_2(\mathbb{C})$$ agrees with $$\Gamma\overset{\rho}{\longrightarrow}\PGL_2(\mathscr{O}_K)\to\PGL_2(\mathbb{C}).$$ Then $\rho$ lifts (up to conjugacy) to a representation $\Gamma\to\SL_2(\mathscr{O}_{K'})$.
\end{lemma}
\begin{proof}
	Choose generators $\gamma_1, \cdots, \gamma_n$ of $\Gamma$. For each $i$, $\rho(\gamma_i)$ lifts to $\SL_2(\mathscr{O}_{K'})$, as $\SL_{2, \mathbb{Z}}\to\PGL_{2,\mathbb{Z}}$ is finite flat of degree $2$, and \'etale away from the prime $2$. The obstruction to choosing such lifts such that one obtains an honest representation of $\pi_1(Y)$ lies in $H^2(\pi_1(Y), \{\pm 1\})$, but it vanishes by the assumption of the existence of $\rho'$.
\end{proof}

\subsection{The proof}
\begin{proof}[Proof of \autoref{thm:main-theorem}]
	As in the introduction, let $Y$ be a smooth complex variety equipped with a smooth projective simple normal crossings compactification $\overline{Y}$, with $D=\overline{Y}\setminus Y$, and $D=\cup_{i=1}^n D_i$ the irreducible components of $D$. Let $G=\PGL_{2, \mathbb{Z}}$ or $G=\SL_{2, \mathbb{Z}}$ and fix points $\underline{C}=(C_1, \cdots, C_n)\in (G/_{\text{ad}}G)^n(\mathscr{O}_K)$. We will show that there exists an extension $K'$ of $K$ such that $\mathscr{O}_{K'}$-points are Zariski-dense in $X_{G, \underline{C}}(Y)$.
	
	We first do this in the case $G=\PGL_{2, \mathbb{Z}}$. Let $W\subset X_{G, \underline{C}}(Y)_{\overline{\mathbb{Q}}}$ be an irreducible component. Let $\eta$ be the generic point of $W$ and $$\rho: \pi_1(Y)\to\PGL_2(\overline{\kappa(\eta)})$$ the corresponding representation. If $\rho$ is not Zariski-dense in $\PGL_2$, then integral points are dense in $W$ by \autoref{lemma:non-dense-reps}. We may thus assume $\rho$ has Zariski-dense image.
	
	Case 1: $W$ is zero-dimensional and $\underline{C}$ is quasi-unipotent. In this case, $W=\{[\rho]\}$, which is integral by \cite[Theorem 7.3]{corlette2008classification}.
	
	Case 2: $W$ is positive-dimensional or $\underline{C}$ is not quasi-unipotent. In this case, $\rho$ factors through a map $Y\to Z$, with $Z$ a orbicurve, by \cite[Theorem 1]{corlette2008classification} in the case $W$ is positive-dimensional and $\underline{C}$ is quasi-unipotent, and by \cite[Theorem A]{loray2016representations}  in the case $\underline{C}$ is not quasi-unipotent. By Stein factorization we may assume this map has connected fibers. The point $[\rho]$ is in the image of the induced map $X_{G, \underline{C}'}(Z)\to X_{G,\underline{C}}(Y)$ for appropriate $\underline{C}'$, by \autoref{lem:integral-conjugacy-class}. But there exists $K'$ such that $\mathscr{O}_{K'}$-points are Zariski-dense in $X_{G, \underline{C}'}(Z)$ by \autoref{theorem:pgl2-density} (here we use that relative character varieties of orbicurves are disjoint unions of relative character varieties of surface groups). Thus $[\rho]$ is in the Zariski-closure of the $\mathscr{O}_{K'}$-points of $W$; as $[\rho]$ is itself Zariski-dense in $W$, the proof is complete.
	
	We now consider the case $G'=\SL_{2, \mathbb{Z}}$; we still write $G=\PGL_{2, \mathbb{Z}}$. Let $\underline{C}'\in (G'/_{\text{ad}}G')^n(\mathscr{O}_K)$ be a tuple, and $\underline{C}$ its image in $(G/_{\text{ad}}G)^n(\mathscr{O}_K)$, and consider the map $X_{G', \underline{C}'}(Y)\to X_{G, \underline{C}}(Y)$. Let $W'$ be an irreducible component of $X_{G', \underline{C}'}(Y)_{\overline{\mathbb{Q}}}$ and $[\rho]$ its generic point. Again if $\rho$ is not Zariski-dense in $\SL_2$, then integral points are dense in $W'$ by \autoref{lemma:non-dense-reps}. 
	
	If $[\rho]$ is Zariski-dense, consider the image $W$ of $W'$ in $X_{G,\underline{C}}(Y)$. $W$ is an irreducible component of $X_{G,\underline{C}}(Y)$, so we have by the previous paragraph that $\mathscr{O}_{K'}$-points are dense in $W$ for some $K'$. After replacing $K'$ by a finite extension as in \autoref{lemma:lifting-reps}, all these points lift to $W'$, as the same is true for $\rho$ (here we use that the cohomological obstruction to lifting is constant on connected components).  As $W'\to W$ is finite (by e.g.~\cite[Theorem 1.1]{cotner2024morphisms}) we thus have that $[\rho]$ is in the closure of the $\mathscr{O}_{K'}$-points of $W'$, and the proof is complete.
\end{proof}

\bibliographystyle{alpha}
\bibliography{bibliography-rank-2}

\begin{thebibliography}{MnMO24}

\bibitem[AG]{golsefidy-complete}
N.~Tamam A.~Golsefidy.
\newblock Closure of orbits of the pure mapping class group in the character
  variety.
\newblock \url{https://mathweb.ucsd.edu/~asalehig/GT_MCG-all.pdf}.

\bibitem[BGS16]{bourgain2016markoff}
Jean Bourgain, Alexander Gamburd, and Peter Sarnak.
\newblock Markoff triples and strong approximation.
\newblock {\em Comptes Rendus. Math{\'e}matique}, 354(2):131--135, 2016.

\bibitem[Cam11]{campana2011special}
Fr\'ed\'eric Campana.
\newblock Orbifoldes g\'eom\'etriques sp\'eciales et classification
  bim\'eromorphe des vari\'et\'es k\"ahl\'eriennes compactes.
\newblock {\em J. Inst. Math. Jussieu}, 10(4):809--934, 2011.

\bibitem[Che24]{chen2024nonabelian}
William~Y Chen.
\newblock Nonabelian level structures, {N}ielsen equivalence, and {M}arkoff
  triples.
\newblock {\em Annals of Mathematics}, 199(1):301--443, 2024.

\bibitem[Cot24]{cotner2024morphisms}
Sean Cotner.
\newblock Morphisms of character varieties.
\newblock {\em International Mathematics Research Notices},
  2024(16):11540--11548, 2024.

\bibitem[CS08]{corlette2008classification}
Kevin Corlette and Carlos Simpson.
\newblock On the classification of rank-two representations of quasiprojective
  fundamental groups.
\newblock {\em Compositio Mathematica}, 144(5):1271--1331, 2008.

\bibitem[DJE24]{de2024integrality}
Johan De~Jong and H{\'e}l{\`e}ne Esnault.
\newblock Integrality of the {B}etti moduli space.
\newblock {\em Transactions of the American Mathematical Society},
  377(01):431--448, 2024.

\bibitem[Dri12]{drinfeld2012conjecture}
Vladimir~Gershonovich Drinfeld.
\newblock On a conjecture of {D}eligne.
\newblock {\em Moscow Mathematical Journal}, 12(3):515--542, 2012.

\bibitem[EG18]{esnault2018cohomologically}
H{\'e}l{\`e}ne Esnault and Michael Groechenig.
\newblock Cohomologically rigid local systems and integrality.
\newblock {\em Selecta Mathematica}, 24(5):4279--4292, 2018.

\bibitem[Gol05]{goldman2005mapping}
William~M Goldman.
\newblock The mapping class group acts reducibly on ${SU}(n)$-character
  varieties.
\newblock {\em arXiv preprint math/0509115}, 2005.

\bibitem[GS22]{ghosh2022hasse}
Amit Ghosh and Peter Sarnak.
\newblock Integral points on {M}arkoff type cubic surfaces.
\newblock {\em Invent. Math.}, 229(2):689--749, 2022.

\bibitem[GT25]{golsefidy2025closure}
Alireza~S Golsefidy and Nattalie Tamam.
\newblock Closure of orbits of the pure mapping class group in the character
  variety.
\newblock {\em Proceedings of the National Academy of Sciences},
  122(15):e2416120122, 2025.

\bibitem[GX09]{goldman2009ergodicity}
William~M Goldman and Eugene~Z Xia.
\newblock Ergodicity of mapping class group actions on ${SU}(2)$-character
  varieties.
\newblock {\em arXiv preprint arXiv:0901.1402}, 2009.

\bibitem[KNPS15]{simpson2015harmonic}
Ludmil Katzarkov, Alexander Noll, Pranav Pandit, and Carlos Simpson.
\newblock Harmonic maps to buildings and singular perturbation theory.
\newblock {\em Comm. Math. Phys.}, 336(2):853--903, 2015.

\bibitem[KP22]{klevdal-patrikis}
Christian Klevdal and Stefan Patrikis.
\newblock {$G$}-cohomologically rigid local systems are integral.
\newblock {\em Trans. Amer. Math. Soc.}, 375(6):4153--4175, 2022.

\bibitem[Laf02]{lafforgue2002chtoucas}
Laurent Lafforgue.
\newblock Chtoucas de {D}rinfeld et correspondance de {L}anglands.
\newblock {\em Inventiones mathematicae}, 147:1--241, 2002.

\bibitem[Lit24]{litt2024motives}
Daniel Litt.
\newblock Motives, mapping class groups, and monodromy.
\newblock {\em arXiv preprint arXiv:2409.02234}, 2024.

\bibitem[LL24]{landesman2024canonical}
Aaron Landesman and Daniel Litt.
\newblock Canonical representations of surface groups.
\newblock {\em Annals of Mathematics}, 199(2):823--897, 2024.

\bibitem[LL25]{lam2025algebraicity}
Yeuk Hay~Joshua Lam and Daniel Litt.
\newblock Algebraicity and integrality of solutions to differential equations.
\newblock {\em arXiv preprint arXiv:2501.13175}, 2025.

\bibitem[LLL23]{lam2023finite}
Yeuk Hay~Joshua Lam, Aaron Landesman, and Daniel Litt.
\newblock Finite braid group orbits on ${SL}_2$-character varieties.
\newblock {\em arXiv preprint arXiv:2308.01376}, 2023.

\bibitem[LMnN13]{logares2013hodge}
Marina Logares, Vicente Mu\~noz, and P.~E. Newstead.
\newblock Hodge polynomials of {${\rm SL}(2,\Bbb{C})$}-character varieties for
  curves of small genus.
\newblock {\em Rev. Mat. Complut.}, 26(2):635--703, 2013.

\bibitem[LPT16]{loray2016representations}
Frank Loray, Jorge~Vit{\'o}rio Pereira, and Fr{\'e}d{\'e}ric Touzet.
\newblock Representations of quasi-projective groups, flat connections and
  transversely projective foliations.
\newblock {\em Journal de l'{\'E}cole polytechnique-Math{\'e}matiques},
  3:263--308, 2016.

\bibitem[Mar25]{martin2025new}
Daniel~E Martin.
\newblock A new proof of {C}hen's theorem for {M}arkoff graphs.
\newblock {\em arXiv preprint arXiv:2502.15960}, 2025.

\bibitem[MnMO24]{munoz2024coordinate}
Vicente Mu\~noz and Jes\'us Mart\'in~Ovejero.
\newblock Coordinate rings of some {$\rm SL_2$}-character varieties.
\newblock {\em Rev. Un. Mat. Argentina}, 67(1):47--64, 2024.

\bibitem[PX00]{previte2000topological}
Joseph~P Previte and Eugene~Z Xia.
\newblock Topological dynamics on moduli spaces, {I}.
\newblock {\em Pacific Journal of Mathematics}, 193(2):397--417, 2000.

\bibitem[PX02]{previte2002topological}
Joseph Previte and Eugene Xia.
\newblock Topological dynamics on moduli spaces {II}.
\newblock {\em Transactions of the American Mathematical Society},
  354(6):2475--2494, 2002.

\bibitem[Sim92]{simpson1992higgs}
Carlos~T Simpson.
\newblock Higgs bundles and local systems.
\newblock {\em Publications Math{\'e}matiques de l'IH{\'E}S}, 75:5--95, 1992.

\bibitem[Wha20a]{whang:ant}
Junho~Peter Whang.
\newblock Arithmetic of curves on moduli of local systems.
\newblock {\em Algebra Number Theory}, 14(10):2575--2605, 2020.

\bibitem[Wha20b]{whang2020global}
Junho~Peter Whang.
\newblock Global geometry on moduli of local systems for surfaces with
  boundary.
\newblock {\em Compositio Mathematica}, 156(8):1517--1559, 2020.

\bibitem[Wha20c]{whang:israel-journal}
Junho~Peter Whang.
\newblock Nonlinear descent on moduli of local systems.
\newblock {\em Israel J. Math.}, 240(2):935--1004, 2020.

\end{thebibliography}

\end{document}